\documentclass[12pt]{amsart}
\usepackage{amssymb,amsmath, amsthm,latexsym,verbatim}
\usepackage{a4wide}
\usepackage{mathrsfs}
\usepackage[hypertex]{hyperref}
\usepackage{enumitem}

\theoremstyle{plain}

\newtheorem{lem}{Lemma}[section]
\newtheorem{theo}[lem]{Theorem}
\newtheorem{prop}[lem]{Proposition}
\newtheorem{cor}[lem]{Corollary}
\newtheorem{remark}[lem]{Remark}

\parskip=\medskipamount
\font\k=cmr7
\font\rm=cmr12

  \newcommand {\di}{\mbox{\k dis}}
  \newcommand {\fin}{\mbox{\k fin}}
  \newcommand{\unip}{\operatorname{unip}}

  \newcommand {\reg}{\mbox{\k reg}}
  
  \newcommand {\spec}{\mbox{\k spec}}
  \newcommand {\geo}{\mbox{\k geo}}

  \newcommand {\C}{{\mathbb C}}
  \newcommand {\bH}{{\mathbb H}}
  \newcommand {\N}{{\mathbb N}}
  \newcommand {\R}{{\mathbb R}}
  \newcommand {\Z}{{\mathbb Z}}
  \newcommand {\Q}{{\mathbb Q}}
  \newcommand {\A}{{\mathbb A}}
  
  \newcommand {\af}{{\mathfrak a}}
  \newcommand {\gf}{{\mathfrak g}}
  
  \newcommand {\kf}{{\mathfrak k}}

  \newcommand {\of}{{\mathfrak o}}
  
  \newcommand {\uf}{{\mathfrak u}}
  
  \newcommand {\pg}{{\mathfrak p}}
   \newcommand {\pf}{{\mathfrak p}}

\renewcommand {\H}{{\mathcal H}}
  \newcommand {\M}{{\mathcal M}}
  \newcommand {\cF}{{\mathcal F}}
  \newcommand {\Co}{{\mathcal C}}
 \newcommand {\cO}{{\mathcal O}}
 \newcommand {\ccC}{{\mathscr C}}
  \newcommand {\G}{{\bf G}}

 \newcommand {\cH}{{\mathcal H}}
 \newcommand {\cP}{{\mathcal P}}
 \newcommand {\cL}{{\mathcal L}}
 \newcommand {\cA}{{\mathcal A}}
 
\newcommand {\cT}{{\mathcal T}}
\newcommand  {\cZ}{{\mathcal Z}}
\newcommand {\bs}{\backslash}

\newcommand{\cU}{{\mathcal U}}

\newcommand{\levis}{{\mathcal L}}
\newcommand{\Ai}{A_M(\R)^0}

\renewcommand{\Re}{\operatorname{Re}}
\newcommand{\Tr}{\operatorname{Tr}}

\newcommand{\End}{\operatorname{End}}

\newcommand{\tr}{\operatorname{tr}}

\newcommand{\inj}{\operatorname{inj}}
\newcommand{\Id}{\operatorname{Id}}
\newcommand{\Hom}{\operatorname{Hom}}

\newcommand{\Ind}{\operatorname{Ind}}

\newcommand{\rk}{\operatorname{rank}}
\newcommand{\vol}{\operatorname{vol}}

\newcommand{\SL}{\operatorname{SL}}
\newcommand{\GL}{\operatorname{GL}}
\newcommand{\SO}{\operatorname{SO}}

\newcommand{\Ad}{\operatorname{Ad}}

\newcommand{\supp}{\operatorname{supp}}

\renewcommand{\det}{\operatorname{det}}

\newcommand{\Rep}{\operatorname{Rep}}

\newcommand{\norm}[1]{\lVert#1\rVert}
\newcommand{\abs}[1]{\lvert#1\rvert}
\newcommand{\eps}{\epsilon}
\newcommand{\One}{\mathbf 1}
\newcommand{\erf}{\operatorname{erf}}

\newcommand{\aaa}{\mathfrak{a}}
  \newcommand {\K}{{\bf K}}
  
  \newcommand {\bU}{{\bf U}}
  
  \newcommand{\Ht}{H}
\newcommand{\sprod}[2]{\left\langle#1,#2\right\rangle}
\newcommand{\PPP}{\mathcal{P}}
\newcommand{\FFF}{{\mathcal F}}
\newcommand{\rts}{\Sigma}

\newcommand{\disc}{\operatorname{disc}}
\newcommand{\srts}{\Delta}
\newcommand{\modulus}{\delta}
\newcommand{\AF}{{\mathcal A}}
\newcommand{\zzz}{\mathfrak{z}}
\newcommand{\iii}{{\mathrm i}}
\newcommand{\LieG}{\mathfrak{g}}
\newcommand{\bases}{\mathfrak{B}}
\newcommand{\bss}{\underline{\beta}}
\newcommand{\dtup}{\mathcal{X}}
\newcommand{\card}[1]{\lvert#1\rvert}

\newcommand{\ka}{\mathfrak{a}}

\newcommand{\diag}{\operatorname{diag}}

\newcommand{\cpt}{\mathbf{K}}
\newcommand{\level}{\operatorname{level}}

\newcommand{\FP}{\operatorname{FP}}
\newcommand{\Mat}{\operatorname{Mat}}

\newcommand{\OOO}{\mathcal{O}}

\newcommand{\1}{{\bf 1}}

\begin{document}

\title[]
{Approximation of $L^2$-analytic torsion for arithmetic quotients of the
symmetric space $\SL(n,\R)/\SO(n)$}
\date{\today}

\author{Jasmin Matz}
\address{The Hebrew University of Jerusalem\\
Einstein Institute of Mathematics\\
Givat Ram \\
Jerusalem 9190401 \\
Israel}
\email{jasmin.matz@mail.huij.ac.il}

\author{Werner M\"uller}
\address{Universit\"at Bonn\\
Mathematisches Institut\\
Endenicher Allee 60\\
D -- 53115 Bonn, Germany}
\email{mueller@math.uni-bonn.de}


\begin{abstract}
In \cite{MzM} we defined a regularized analytic torsion for 
quotients of the symmetric space $\SL(n,\R)/\SO(n)$ by arithmetic lattices. In 
this paper we study the limiting behaviour of the analytic torsion as the 
lattices run through sequences of congruence subgroups of a fixed arithmetic 
subgroup. Our main result states that for principal congruence subgroups and
strongly acyclic flat bundles, the logarithm of the analytic torsion, devided
 by the index of the subgroup, converges to the $L^2$-analytic torsion. 
\end{abstract}

\maketitle
\setcounter{tocdepth}{1}
\tableofcontents

\section{Introduction}

Let $X$ be a compact oriented Riemannian manifold of dimension $d$. Let
 $\rho$ be a finite
dimensional representation of $\pi_1(X)$ and let $E_\rho\to X$ be the
associated flat vector bundle. Pick a Hermitian fiber metric in $E_\rho$.
 Let $\Delta_p(\rho)$
be the Laplace operator on $E_\rho$-valued $p$-forms. Let $\zeta_p(s,\rho)$
be its zeta function \cite{Sh}. Let $e^{-t\Delta_p(\rho)}$, $t>0$, be the heat
operator and let $b_p(\rho)=\dim\ker\Delta_p(\rho)$. Then for  $\Re(s)>d/2$ 
one has
\begin{equation}\label{mellin-transf}
\zeta_p(s,\rho)=\frac{1}{\Gamma(s)}\int_0^\infty (\Tr\left(e^{-t\Delta_p(\rho)}
\right)-b_p(\rho)) t^{s-1} dt.
\end{equation}
Then the analytic torsion $T_X(\rho)\in\R^+$, introduced by Ray and Singer 
\cite{RS}, is defined by
\begin{equation}\label{analtors1}
\log T_X(\rho):=\frac{1}{2}\sum_{p=1}^d (-1)^pp\frac{d}{ds}\zeta_p(s;\rho)
\big|_{s=0}.
\end{equation}
The corresponding $L^2$-invariant, the $L^2$-analytic torsion $T^{(2)}_X(\rho)$,
 was introduced by Lott \cite{Lo} and Mathai \cite{MV}. It is defined in terms
of the von Neumann trace of the heat operators on the universal covering 
$\widetilde X$ of $X$. 

The analytic torsion has been  used by Bergeron and Venkatesh \cite{BV} to 
study the growth of torsion in the cohomology of cocompact arithmetic groups.
The approach of \cite{BV} is based on the approximation of the $L^2$-torsion
by the renormalized analytic torsion for sequences of coverings of a given
compact locally symmetric space. Since many important arithmetic groups are
not cocompact, it is desirable to extend these results to the non-compact 
case. The first problem is that the analytic torsion is not defined for 
non-compact manifolds. To cope with this problem we defined in \cite{MzM} 
a regularized version of the analytic torsion for
quotients of the symmetric space $\SL(n,\R)/\SO(n)$ by arithmetic groups.
The goal of the present paper is to extend the result of Bergeron and 
Venkatesh \cite{BV} on the approximation of the $L^2$-analytic torsion
to this setting. 

To begin with we recall the results of Bergeron and Venkatesh. 
Let $G$ be a semisimple Lie group of non-compact type. Let $K$ be a maximal
compact subgroup of $G$ and let $\widetilde X=G/K$  be the associated Riemannian
symmetric space endowed with a $G$-invariant metric. Let $\Gamma\subset G$ be 
a cocompact discrete subgroup.
For simplicity we assume that $\Gamma$ is torsion free. Let $X:=\Gamma\backslash
\widetilde X$. Then $X$ is a compact locally symmetric manifold of non-positive
curvature. Let $\tau$ be an irreducible  finite dimensional complex
representation of $G$. Denote by $T_X(\tau)$ (resp. $T^{(2)}_X(\tau)$)
the analytic torsion (resp. the $L^2$-torsion) taken with respect
to the representation $\tau|_\Gamma$ of $\Gamma$. 
Since the heat kernels on $\widetilde X$ are $G$-invariant, one has
\begin{equation}\label{l2-tor}
\log T^{(2)}_X(\tau)=\vol(X) t^{(2)}_{\widetilde X}(\tau),
\end{equation}
where $t^{(2)}_{\widetilde X}(\tau)$ is a constant that depends only on 
$\widetilde X$ and $\tau$. 
It is an interesting problem to see if the $L^2$-torsion can be approximated
by the torsion of finite coverings $X_i\to X$. This problem has been studied
by Bergeron and Venkatesh \cite{BV} under a certain non-degeneracy condition
on $\tau$. Representations which satisfy this condition are called 
{\it strongly acyclic}. One of the main results of \cite{BV} is as follows. 
Let $X_i\to X$, $i\in\N$, be a sequence of finite coverings of $X$. 
Let $\tau$ be strongly acyclic. Let $\inj(X_i)$ denote the injectivity radius
of $X_i$ and assume that $\inj(X_i)\to\infty$ as $i\to\infty$. Then
by \cite[Theorem 4.5]{BV} one has
\begin{equation}\label{limittor}
\lim_{i\to\infty}\frac{\log T_{X_i}(\tau)}{\vol(X_i)}=t^{(2)}_{\widetilde
X}(\tau).
\end{equation}
Let $\delta(\widetilde X):=\rk_\C(G)-\rk_\C(K)=1$. The constant 
$t^{(2)}_{\widetilde X}(\rho)$ has been computed by Bergeron and Venkatesh
\cite[Proposition 5.2]{BV}. It is shown that $t^{(2)}_{\widetilde X}(\rho)\neq0$
if and only if $\delta(\widetilde X)=1$. 
Combined with the equality of analytic torsion and Reidemeister torsion 
\cite{Mu2}, Bergeron and Venkatesh \cite{BV} used this result in the
case $\delta(\widetilde X)=1$ to study the 
growth of torsion in the cohomology of cocompact arithmetic groups. 
Unfortunately, so far the method does not work for representations which are
not strongly acyclic. Especially, it does not work for the trivial 
representation, which is the most interesting case.  For a detailed
discussion of this problem in the case of hyperbolic 3-manifolds see 
\cite{BSV}. 

Another challenging problem is to extend the method to the case of arithmetic 
lattices which are not cocompact.
In \cite{AY} Ash, Gunnells, McConnell and Yasaki investigated the growth
of torsion in the cohomology of non-cocompact arithmetic subgroups 
$\Gamma\subset\GL(n,\Z)$ in the case of the trivial coefficient system,
and formulated a number of conjectures concerning the expected behavior of
torsion cohomology. To study the growth of the 
torsion in the cohomology of non-cocompact arithmetic groups  one can try to 
proceed as in \cite{BV}. As a first step one would like to extend 
\eqref{limittor} to the finite volume case. 
However, due to the presence of the continuous spectrum of the Laplace 
operators in the non-compact case, one encounters serious technical 
difficulties in attempting
to generalize  \eqref{limittor} to the finite volume case. In \cite{Ra1} J. 
Raimbault has dealt with finite volume hyperbolic 3-manifolds. In \cite{Ra2}
he applied this to study the growth of torsion in the cohomology for certain
sequences of congruence subgroups of Bianchi groups. 

The main
purpose of the present paper is to extend \eqref{limittor} to arithmetic 
quotients of 
\[
\widetilde X:=\SL(n,\R)/\SO(n).
\]
The regularized analytic torsion in the non-compact case has been defined in
\cite{MzM}. For its definition we pass to the 
adelic framework. Let $G=\SL(n)$. Let $\A$ be the ring of adeles and $\A_f$ the 
ring of finite adeles. Let $K_\infty=\SO(n)$ be the usual maximal compact 
subgroup of $G(\R)=\SL(n,\R)$.  Given  an open compact subgroup, 
$K_f\subset G(\A_f)$, let
\begin{equation}\label{adelic-mfd}
X(K_f):= G(\Q)\bs (\widetilde X\times G(\A_f)/K_f)
\end{equation}
be the associated adelic quotient. This is the adelic version of a locally
symmetric space.  Since $\SL(n)$ is simply connected,  strong approximation
holds for $\SL(n)$ and therefore, we have
\begin{equation}
X(K_f)=\Gamma\bs\widetilde X,
\end{equation}
where $\Gamma$ is the projection of $(G(\R)\times K_f)\cap G(\Q)$ onto $G(\R)$.
We will assume that $K_f$ is neat
so that $X(K_f)$ is a manifold. Let $\tau\colon G(\R)\to \GL(V_\tau)$ be a 
finite dimensional complex representation. The restriction of $\tau$ to 
$\Gamma\subset G(\R)$ induces a flat 
vector bundle $E_{\tau}$ over $X(K_f)$. By \cite{MM}, $E_{\tau}$
is isomorphic to the locally homogeneous vector bundle over $X(K_f)$, which is 
associated to $\tau|_{K_\infty}$. Moreover it can be
equipped with a distinguished fiber metric, induced from an admissible inner
product in $V_\tau$. In this way we get a fiber metric in $E_\tau$. 
Let $\Delta_p(\tau)$ be the  twisted Laplace operator on
$p$-forms with values in $E_\tau$. If $X(K_f)$ is not compact, $\Delta_p(\tau)$
has continuous spectrum and therefore, the analytic torsion can not be defined
by \eqref{analtors1}. In \cite{MzM} we have introduced a regularized version 
of the analytic torsion.  The starting point for the
definition of the regularized analytic torsion in the non-compact case is
formula \eqref{mellin-transf}. In
\cite{MzM} we introduced a regularized trace of the heat operator. It is 
defined as follows. Let $\widetilde\Delta_p(\tau)$ be the Laplace operator
on $\widetilde E_\tau$-valued $p$-forms on $\widetilde X$. The heat operator
$e^{-\widetilde \Delta_p(\tau)}$ is a convolution operator given by a kernel
$H_t^{\tau,p}\colon G(\R)\to \GL(\Lambda^p\pg^\star\otimes V_\tau)$, where
$\gf=\kf\oplus\pg$ is the Cartan decomposition of the Lie algebra $\gf$ of
$G(\R)$. Let $h_t^{\tau,p}\in C^\infty(G(\R))$ be defined by
\[
h_t^{\tau,p}(g)=\tr H_t^{\tau,p}(g),\quad g\in G(\R).
\]
Let $J_{\geo}(f)$, $f\in C_c^\infty(G(\A))$, be the geometric side of the 
(non-invariant) Arthur trace formula \cite{Ar1}. By \cite[Theorem 7.1]{FL1},
$J_{\geo}(f)$ is defined for all $f\in\Co(G(\A),K_f)$, the adelic version of the
Schwartz space (see section \ref{sec-prel} for its definition).  Let $\1_{K_f}$ 
be the characteristic function of $K_f$ in $G(\A_f)$. Put
\begin{equation}\label{normal-chfct}
\chi_{K_f}:=\frac{\1_{K_f}}{\vol(K_f)}.
\end{equation}
Then $h_t^{\tau,p}\otimes \chi_{K_f}$ belongs to the Schwartz space 
$\Co(G(\A),K_f)$, and in 
\cite[(13.16)]{MzM} we defined the regularized trace of the heat operator by
\begin{equation}\label{regtrace}
\Tr_{\reg}\left(e^{-t\Delta_p(\tau)}\right)=J_{\geo}(h_t^{\tau,p}\otimes \chi_{K_f}).
\end{equation}
If $X(K_f)$ is compact, this equality is just the content of the trace formula. 
For the motivation of this definition see \cite{MzM}.

In order to be able to 
use the Mellin transform to define a regularized zeta function similar to 
\eqref{mellin-transf} one needs to know the asymptotic behavior of the 
regularized
trace of the heat operator
as $t\to\infty$ and $t\to 0$. Let $\theta$ be the Cartan involution of 
$G(\R)$. Let $\tau_\theta:=\tau\circ\theta$. Assume that 
$\tau\ncong\tau_\theta$. Then by \cite[Theorem 1.2]{MzM} there exists $c>0$
such  that $\Tr_{\reg}\left(e^{-t\Delta_p(\tau)}\right)=O(e^{-ct})$ as $t\to\infty$ 
for all $p=0,\dots,d$. Furthermore, by \cite[Theorem 1.1]{MzM}, 
$\Tr_{\reg}\left(e^{-t\Delta_p(\tau)}\right)$ admits an asymptotic expansion as
$t\to 0$. This expansion contains logarithmic terms. Using these facts,
the zeta function $\zeta_p(s,\tau)$ can be defined as in \eqref{mellin-transf}
with the trace of the heat operator replaced by the regularized trace. Due to
the presence of log-terms in the asymptotic expansion for $t\to 0$, 
$\zeta_p(s,\tau)$ may have a pole at $s=0$. So the definition  \eqref{analtors1}
of the analytic torsion has to be modified. 
Let $f(s)$ be a meromorphic function on $\C$. For $s_0\in\C$ 
let $f(s)=\sum_{k\ge k_0}a_k(s-s_0)^k$
be the Laurent expansion of $f$ at $s_0$. Put $\FP_{s=s_0}f(s):=a_0$. 
Now we define
the analytic torsion $T_{X(K_f)}(\tau)\in\C\setminus\{0\}$ by
\begin{equation}\label{analtor}
\log T_{X(K_f)}(\tau)=\frac{1}{2}\sum_{p=0}^d (-1)^p p 
\left(\FP_{s=0}\frac{\zeta_p(s;\tau)}{s}\right).
\end{equation} 
If the zeta functions are holomorphic at $s=0$, this is the same definition as
before. 

Now we can formulate our main result. Let $n\ge 2$. Put
$\widetilde X_n=\SL(n,\R)/\SO(n)$. Let $K_n(N)\subset\SL(n,\A_f)$ be the 
principal congruence subgroup of level $N\ge 3$. Put $X_n(N):=X(K_n(N))$. 
 Note that $X_n(N)=\Gamma(N)\bs\widetilde X_n$, where 
$\Gamma(N)\subset\SL(n,\Z)$ is the 
principal congruence subgroup of level $N$. 
Then our main result is the following theorem
\begin{theo}\label{theo-main}
Let $\tau\in\Rep(\SL(n,\R))$. Assume that $\tau\ncong\tau_\theta$. Then for
$n\ge 2$ we have
\[
\lim_{N\to\infty}\frac{\log T_{X_n(N)}(\tau)}{\vol(X_n(N))}
=t^{(2)}_{\widetilde X_n}(\tau).
\]
Moreover, if $n>4$, then $t^{(2)}_{\widetilde X_n}(\tau)=0$, and if $n=3,4$, then
$t^{(2)}_{\widetilde X_n}(\tau)>0$.
\end{theo}

\begin{remark}
The number $t^{(2)}_{\widetilde X}(\rho)$ can be defined for every 
finite dimensional representation (cf. \cite[4.4]{BV}). Moreover, it can be
computed explicitly \cite[\S5]{BV} (see also \cite[\S3]{AY}). For example, for the trivial representation
$\tau_0$ of $\SL(n,\R)$, $n=3,4$, one has
\[
t^{(2)}_{\widetilde X_3}(\tau_0)=\frac{\pi}{2\vol(\widetilde X_3^c)}, \quad 
t^{(2)}_{\widetilde X_4}(\tau_0)=\frac{124\pi}{45 \vol(\widetilde X_4^c)}
\]
\cite[5.9.3, Example 2]{BV}. Here $\widetilde X_j^c$ denotes the compact dual of $\widetilde X_j$, and the metric on $\widetilde X_j^c$ is the one induced from the metric on $\widetilde X_j$.  For the second equality we used that $\SL(4,\R)$
is a double covering of $\SO(3,3)$, and as explained at the beginning of
section 5.8 in \cite{BV}, the corresponding number for $\SO(3,3)$ agrees with
that for $\SO(5,1)$. Finally, $t^{(2)}_{\bH^5}(\tau_0)$ is computed in 
\cite[5.9.3, Example 1]{BV}.
\end{remark}
\begin{remark}
Let $\Gamma\subset\SL(n,\R)$ be a cocompact torsion free lattice. Then 
$T_{\Gamma\bs \widetilde X_n}(\tau)=1$ for all $n>4$ and all 
$\tau\in\Rep(\SL(n,\R))$.
This follows in exactly the same way as in \cite[Corollary 2.2]{MS}. We don't
know if this also holds in the non-cocompact case. 
\end{remark}
\begin{remark}
We expect that Theorem \ref{theo-main} holds more generally for sequences of 
arbitrary congruence quotients $Y_j=\Gamma_j\bs\SL_n(\R)/\SO(n)$ such that 
$\vol(Y_j)\to\infty$ as $j\to\infty$. The extension hinges at the solution of
some technical problems related to the fine geometric expansion of the trace
for $\SL(n)$. For more details see the end of the section.
\end{remark}

We shall now briefly outline our method to prove Theorem \ref{theo-main}.
For technical reasons we work with $\GL(n)$ in place of $\SL(n)$. 
Let $K_f\subset\GL(n,\A_f)$ be an open compact subgroup. Then we define the
corresponding adelic quotient $Y(K_f)$ as above by
\[
Y(K_f):=\GL(n,\Q)\bs(\widetilde X\times\GL(n,\A_f))/K_f.
\]
We note that $Y(K_f)$ is the disjoint union of finitely many locally symmetric
spaces $\Gamma_i\bs\widetilde X$ for arithmetic subgroups $\Gamma_i\subset 
\GL(n,\Q)$, $i=1,\dots,l$. Now let $K(N)\subset\GL(n,\A_f)$ be the principal 
congruence subgroup
of level $N$. Put $Y(N):=Y(K(N))$. Then $Y(N)$ is the disjoint union of 
$\varphi(N)$ copies of $X(N)$, where $\varphi(N)$ is Euler's function (see
\cite[p. 13]{Ar6}). The disjoint union of
$\varphi(N)$ copies of the flat $E_\tau$ over $X(N)$ is a flat bundle 
$\widehat E_\tau$ over $Y(N)$. Let $\Delta_{p,N}(\tau)$ be the Laplace operator
on $\widehat E_\tau$-valued $p$-forms on $Y(N)$. We define the regularized 
trace of the heat operator $e^{-t\Delta_{p,N}(\tau)}$ as above by
\[
\Tr\left(e^{-t\Delta_{p,N}(\tau)}\right):=
J_{\geo}^{\GL(n)}(h_t^{\tau,p}\otimes \chi_{K(N)}),
\]
where $J_{\geo}^{\GL(n)}$ is now the geometric side of the trace formula for 
$\GL(n,\A)^1$ and $\chi_{K(N)}$ the normalized characteristic function of 
$K(N)$ in $\GL(n,\A_f)$. Using the regularized trace, we define the analytic
torsion $T_{Y(N)}(\tau)$ in the same way as above. Comparing the trace formulas
for $\SL(n)$ and $\GL(n)$, it follows that
\[
\log T_{Y(N)}(\tau)=\varphi(N) \log T_{X(N)}(\tau).
\]
Furthermore note that $\vol(Y(N))=\varphi(N)\vol(X(N))$. Hence it suffices to
show that
\begin{equation}\label{limtor7}
\lim_{N\to\infty}\frac{\log T_{Y(N)}(\tau)}{\vol(Y(N))}= t^{(2)}_{\widetilde X}(\tau).
\end{equation} 
To establish \eqref{limtor7} we proceed as follows. Let
\begin{equation}
K_N(t,\tau):=\frac{1}{2}\sum_{p=1}^d (-1)^p p 
\Tr_{\reg}\left(e^{-t\Delta_{p,N}(\tau)}\right).
\end{equation}
As observed above, $K_N(t,\tau)$ is exponentially decreasing as $t\to\infty$
and admits an asymptotic expansion as $t\to 0$. Thus the analytic torsion can
be defined by
\begin{equation}\label{anator4}
\log T_{Y(N)}(\tau)=
\FP_{s=0}\left(\frac{1}{s\Gamma(s)}\int_0^\infty
\Tr_{\reg}\left(e^{-t\Delta_{p,N}(\tau)}\right) t^{s-1} dt\right).
\end{equation}
Let $T>0$. We decompose the integral into the integrals over $[0,T]$
and $[T,\infty)$. The integral over $[T,\infty)$ is an entire function of $s$. 
Hence it follows that 
\begin{equation}\label{anator5}
\log T_{Y(N)}(\tau)
=\FP_{s=0}\left(\frac{1}{s\Gamma(s)}
\int_0^T K_N(t,\tau) t^{s-1}dt\right)+\int_T^\infty K_N(t,\tau) t^{-1}dt.
\end{equation}
To deal with the second integral, we show that there exist $C,c>0$ such that
\begin{equation}\label{est-larget7}
\frac{1}{\vol(Y(N))}\left|\Tr_{\reg}\left(e^{-t\Delta_{p,N}(\tau)}\right)\right|\le
C e^{-ct}
\end{equation}
for all $t\ge 1$, $p=0,\dots,d$, and $N\in\N$. To prove \eqref{est-larget7}
we use the definition \eqref{regtrace}  and the trace formula, which gives
\[
\Tr_{\reg}\left(e^{-t\Delta_{p,N}(\tau)}\right)=J_{\spec}(h_t^{\tau,p}\otimes \chi_{K(N)}).
\]
To estimate the right hand side we use the fine spectral expansion of
\cite{FLM1} and proceed as in \cite{MzM}. However, the important new feature is
that we need to control the dependence on $N$ of all constants appearing 
in the estimations. The main ingredients of the spectral side of the trace
formula are logarithmic derivatives of intertwining operators. Uniform 
estimations in $N$ of the relevant integrals containing the logarithmic
derivatives were obtained in \cite{FLM2}. These are essential for our purpose. 
Using \eqref{est-larget7} it follows that $\vol(Y(N))^{-1}$ times the second
integral in \eqref{anator5} is $O(e^{-cT})$, where the implied constants are
independent of $N$. 

To deal with the first term, we first show that, up to a term which is
$O(e^{-cT})$, we can replace $h_t^{\tau,p}$ by a function with compact support
$h_{t,T}^{\tau,p}$ with support depending on $T$ and which coincides with
$h_t^{\tau,p}$ in a neighborhood of $1\in G(\R)^1$. The proof of this result
uses again the fine expansion of the spectral side of the trace formula.
Next we use the geometric side of the trace formula. Let $J_{\unip}$ be
the unipotent contribution to the geometric side. Since $h_{t,T}^{\tau,p}$
has compact support, it follows that for sufficiently large $N$, the geometric
side equals $J_{\unip}(h_{t,T}^{\tau,p}\otimes\chi_{K(N)})$. Next we apply the
fine geometric expansion of \cite{Ar4}, which expresses 
$J_{\unip}(h_{t,T}^{\tau,p}\otimes\chi_{K(N)})$ as a finite sum of weighted orbital
integrals $J_M(\cO,h_{t,T}^{\tau,p}\otimes\chi_{K(N)})$ (see \eqref{fine-exp4}).
Here $M\in\cL$ and $\cO$ runs over the set of unipotent elements in $M(\Q)$ up to 
$M(\Q_S)$-conjugacy for $S=S(N)$ a suitable finite set of places. (If $G=\GL(n)$, the resulting equivalence classes are just the unipotent $M(\Q)$-conjugacy classes in $M(\Q)$.) The coefficients $a^M(S(N),\cO)$ appearing in the fine geometric
expansion depend on a sufficiently large set $S(N)$ of places of $\Q$. Then by
the decomposition formula \eqref{decomp1} for weighted orbital integrals, the
study of $J_M(\cO,h_{t,T}^{\tau,p}\otimes\chi_{K(N)})$ can be reduced to the study
of weighted orbital integrals at infinite place and at the finite places in $S(N)$. 
At the infinite place the weighted orbital integrals are of the form
$J_M^L(\cO_\infty, (h_{t,T}^{\tau,p})_Q)$, where $L\in\cL(M)$, $Q$ is a parabolic
subgroup of $G$ with Levi component $L$, and $(h_{t,T}^{\tau,p})_Q$ is defined by
\eqref{ct}. These integrals have been studied in
\cite{MzM}. By \cite[Proposition 12.3]{MzM}, $J_M^L(\cO_\infty, (h_{t,T}^{\tau,p})_Q)$
has an asymptotic expansion as $t\to 0$. So we can form its partial Mellin
transform \eqref{partial-mellin}, which is a meromorphic function of $s\in\C$. 
Then the constant term in the Laurent expansion is 
the contribution of $J_M^L(\cO_\infty, (h_{t,T}^{\tau,p})_Q)$ to the first term on
the right hand side of \eqref{anator5}. It is just a constant depending on 
$T$, but not $N$. We are left with the finite orbital integrals
$J_M^L(\cO_{\fin},(\chi_{K(N)})_Q)$. Again using the decomposition formula, the
study of these integrals can be reduced to study of integrals of the form
$J_M^{L_p}(\cO_p,\1_{K(N)_p,Q_p})$ at primes $p|N$. Now the point is that in the
case of $\GL(n)$ these integrals can be written as integrals  
over $N_p(\Q_p)$ with a certain weight factor, where $N_p$ is the unipotent 
radical of some parabolic subgroup in $L_p$ (see \eqref{local-int}). The analysis of the weight factors
leads to an estimation of these integrals, depending on $N$. For $M\neq G$ or $M=G$ and $\cO\neq1$,
they all decay in $N$ like $O(N^{-(n-1)}(\log N)^a)$ for some fixed $a>0$.
The final step is to estimate the constants $a^M(S(N),\cO)$ appearing in
the fine geometric expansion \eqref{fine-exp4}. For $\GL(n)$ such estimations
were obtained in \cite{Ma}. The final 
result is that  the contribution to first term of the right hand side of
\eqref{anator5} of the weighted orbital integrals 
$J_M(\cO,h_{t,T}^{\tau,p}\otimes\chi_{K(N)})$ with $M\neq Q$ times $\vol(Y(N))^{-1}$ 
decays like $N^{-(n-1)}(\log N)^a$ for some $a>0$ independent of $N$. For the
contribution of $(G,1)$ we get $\vol(Y(N))(t^{(2)}_{\widetilde X}(\tau)+ O(e^{-cT})$
This completes the proof of the first part of Theorem \ref{theo-main}. The
second statement follows from \cite[Proposition 5.2]{BV}.

We expect that Theorem \ref{theo-main} holds more general for congruence 
subgroups of classical groups. The main obstacle to extend the theorem to 
other groups is
the fine geometric expansion. At the moment, we only know how to estimate 
the coefficients $a^M(S(N),U)$ for $\GL(n)$. Nevertheless, we expect to be
able to overcome this problem. Therefore, we will 
work in each section with the most general assumptions possible. 

The paper is organized as follows. In section \ref{sec-prel} we fix notations
and recall some basic facts. In section \ref{sec-bochlapl} we state some
facts concerning heat kernels on symmetric spaces. In section \ref{sec-analtor}
we recall the definition of the regularized trace of the heat operator on
$Y(K_f)$ and we introduce the analytic torsion. In section \ref{sec-specside}
we review the refined expansion of the spectral side of the Arthur trace 
formula. The spectral side of the trace formula is used in section 
\ref{sec-large-time} to study  the large time
behavior of the regularized trace of the heat operator. The main point is to
derive estimations which are uniform in $K_f$. In section \ref{sec-short-time}
we study the behavior of the regularized trace as $t\to 0$. We use again the
spectral side of the trace formula to show  that, up to an exponentially  
decreasing term, we can replace the heat kernel by a compactly supported 
function. In section \ref{sec-geomside} we use the geometric side, applied to
the modified test function. It turns out that for principal congruence
subgroups $K(N)$ of sufficient high level $N\in\N$, only the unipotent 
contribution to the geometric side occurs. Then we use Arthur's fine
geometric expansion, which expresses the unipotent contribution in terms of
weighted orbital integrals. In section \ref{sec-padic} we derive estimations
for $p$-adic weighted orbital integrals. In the section \ref{sec-main-gln}
we prove our main result $\GL(n)$. Based on this result, we prove Theorem 
\ref{theo-main} in the final section \eqref{sec-main-sln}.

\noindent
{\bf Acknowledgment.} Section \ref{sec-main-sln} is due to Werner Hoffmann.
The authors are very grateful to him for the permission to include it in the
present paper.

\section{Preliminaries}\label{sec-prel}

Let $G$ be a reductive algebraic group defined
over $\Q$. We fix a minimal parabolic subgroup $P_0$ of $G$ 
defined over $\Q$ and a Levi decomposition $P_0=M_0\cdot N_0$, both defined 
over $\Q$. If $G=\GL(n)$, we choose $P_0$ to be the subgroup of upper 
triangular matrices of $G$, $N_0$ 
its unipotent radical, and $M_0$ the group of diagonal matrices in $G$.

Let $\cF$ be the set of parabolic subgroups of $G$ which contain 
$M_0$ and are defined over $\Q$. Let $\cL$ be the set of subgroups of $G$ 
which contain $M_0$ and are Levi components of groups in $\cF$. 
For any $P\in\cF$ we write
\[
P=M_PN_P,
\]
where $N_P$ is the unipotent radical of $P$ and $M_P$ belongs to $\cL$. 
Let $M\in\cL$. Denote by $A_M$ the $\Q$-split component of the center of $M$. 
Put $A_P=A_{M_P}$. Let $L\in\cL$ and assume that $L$ contains $M$. Then $L$ is
a reductive group defined over $\Q$ and $M$ is a Levi subgroup of $L$. We 
shall denote the set of Levi subgroups of $L$ which contain $M$ by $\cL^L(M)$.
We also write $\cF^L(M)$ for the set of parabolic subgroups of $L$, defined 
over $\Q$, which contain $M$, and $\cP^L(M)$ for the set of groups in $\cF^L(M)$
for which $M$ is a Levi component. Each of these three sets is finite. If 
$L=G$, we shall usually denote these sets by $\cL(M)$, $\cF(M)$ and $\cP(M)$.

Let $X(M)_\Q$ be the group of characters of $M$ which are defined over $\Q$. 
Put
\begin{equation}\label{liealg}
\af_{M}:=\Hom(X(M)_\Q,\R).
\end{equation}
This is a real vector space whose dimension equals that of $A_M$. Its dual 
space is
\[
\af_{M}^\ast=X(M)_\Q\otimes \R.
\]
 We shall write, 
\begin{equation}\label{liealg1}
\af_P=\af_{M_P},\;A_0=A_{M_0}\quad\text{and}\quad \af_0=\af_{M_0}.
\end{equation}
For $M\in\cL$ let $A_M(\R)^0$ be the connected component of the identity of
the group $A_M(\R)$. 
Let $W_0=N_{\G(\Q)}(A_0)/M_0$ be the Weyl group of $(G,A_0)$,
where $N_{G(\Q)}(H)$ is the normalizer of $H$ in $G(\Q)$.
For any $s\in W_0$ we choose a representative $w_s\in G(\Q)$.
Note that $W_0$ acts on $\levis$ by $sM=w_s M w_s^{-1}$. For $M\in\cL$ let
$W(M)=N_{\G(\Q)}(M)/M$, which can be identified with a subgroup of $W_0$.

For any $L\in\cL(M)$ we identify $\af_L^\ast$ with a subspace of $\af_M^\ast$.
We denote by $\af_M^L$ the annihilator of $\af_L^\ast$ in $\af_M$. 
We set
\[
\levis_1(M)=\{L\in\levis(M):\dim\aaa_M^L=1\}
\]
and
\begin{equation}\label{f1}
\cF_1(M)=\bigcup_{L\in\levis_1(M)}\cP(L).
\end{equation}
We shall denote the simple roots of $(P,A_P)$ by $\Delta_P$. They are
elements of $X(A_P)_\Q$ and are canonically embedded in $\af_P^\ast$. Let
$\Sigma_P\subset \af_P^\ast$ be the set of reduced roots of $A_P$ on the Lie
algebra of $G$. For any $\alpha\in\rts_M$ we denote by $\alpha^\vee\in\aaa_M$
the corresponding co-root. Let $P_1$ and $P_2$ be parabolic subgroups with
$P_1\subset P_2$. Then $\af_{P_2}^\ast$ is embedded into $\af_{P_1}^\ast$, while
$\af_{P_2}$ is a natural quotient vector space of $\af_{P_1}$. The group
$M_{P_2}\cap P_1$ is a parabolic subgroup of $M_{P_2}$. Let $\Delta_{P_1}^{P_2}$
denote the set of simple roots of $(M_{P_2}\cap P_1,A_{P_1})$. It is a subset
of $\Delta_{P_1}$. For a parabolic subgroup $P$ with $P_0\subset P$ we write
$\Delta_0^P:=\Delta_{P_0}^P$.

Let $\A$ be the ring of adeles of 
$\Q$ and $\A_{\fin}$ the ring of finite adeles of $\Q$.  We fix a maximal 
compact subgroup
$\K=\prod_\nu \K_\nu=K_\infty \K_{\fin}$ of $G(\A)=G(\R)G(\A_{\fin})$. 
We assume that the maximal 
compact subgroup $\K \subset G(\A)$ is admissible with respect to 
$M_0$ \cite[\S 1]{Ar5}.

Let $\Ht_M: M(\A)\rightarrow\aaa_M$ be the 
homomorphism given by
\begin{equation}\label{homo-M}
e^{\sprod{\chi}{\Ht_M(m)}}=\abs{\chi (m)}_\A = \prod_v\abs{\chi(m_v)}_v
\end{equation}
for any $\chi\in X(M)$. Let
\[
M(\A)^1 :=\{m\in M(\A)\colon \Ht_M(m)=0\}. 
\]
Let $\gf$ and $\kf$ denote the Lie algebras of $G(\R)$ and $K_\infty$,
respectively. Let $\theta$ be the Cartan involution of $G(\R)$ with respect to
$K_\infty$. It induces a Cartan decomposition $\mathfrak{g}= 
\mathfrak{p} \oplus \mathfrak{k}$. 
We fix an invariant bi-linear form $B$ on $\mathfrak{g}$ which is positive 
definite on $\mathfrak{p}$ and negative definite on $\mathfrak{k}$.
This choice defines a Casimir operator $\Omega$ on $G(\R)$,
and we denote the Casimir eigenvalue of any $\pi \in \Pi (G(\R))$ by 
$\lambda_\pi$. Similarly, we obtain
a Casimir operator $\Omega_{K_\infty}$ on $K_\infty$ and write $\lambda_\tau$ for 
the Casimir eigenvalue of a
representation $\tau \in \Pi (K_\infty)$ (cf.~\cite[\S 2.3]{BG}).
The form $B$ induces a Euclidean scalar product $(X,Y) = - B (X,\theta(Y))$ on 
$\mathfrak{g}$ and all its subspaces.
For $\tau \in \Pi (K_\infty)$ we define $\norm{\tau}$ as in 
\cite[\S 2.2]{CD}. Note that the restriction of the scalar product 
$(\cdot,\cdot)$ on $\gf$ to $\af_0$ gives $\af_0$ the structure of a 
Euclidean space. In particular, this fixes Haar measures on the spaces 
$\af_M^L$ and their duals $(\af_M^L)^\ast$. We follow Arthur in the 
corresponding normalization of Haar measures on the groups $M(\A)$ 
(\cite[\S 1]{Ar1}). 

Finally we introduce the space of Schwartz functions
$\Co(G(\A)^1)$ from \cite{FL1}. For any compact open subgroup
$K_f$ of $G(\A_f)$ the space $G(\A)^1/K_f$ is the countable disjoint union of
copies of 
\begin{equation}\label{gr1}
G(\R)^1=G(\R)\cap G(\A)^1
\end{equation}
and therefore, it is a differentiable
manifold. Any element $X\in\mathcal{U}(\gf^1_\infty)$ of the universal 
enveloping algebra of the Lie algebra $\gf_\infty^1$ of $G(\R)^1$ defines a
left invariant differential operator $f\mapsto f\ast X$ on $G(\A)^1/K_f$. Let
$\Co(G(\A)^1;K_f)$ be the space of smooth right $K_f$-invariant functions on
$G(\A)^1$ which belong, together with all their derivatives, to $L^1(G(\A)^1)$.
The space $\Co(G(\A)^1;K_f)$ becomes a Fr\'echet space under the seminorms
\[
\|f\ast X\|_{L^1(G(\A)^1)},\quad X\in\mathcal{U}(\gf^1_\infty).
\]
Denote by $\Co(G(\A)^1)$ the union of the spaces $\Co(G(\A)^1;K_f)$ as $K_f$ 
varies over the compact open subgroups of $G(\A_f)$ and endow 
$\Co(G(\A)^1)$ with the inductive
limit topology.

\section{Heat kernels}\label{sec-bochlapl}
\setcounter{equation}{0}

Since the heat kernel of the twisted Laplace operators plays a key role in
the paper, we summarize some basic facts about Bochner-Laplace operators
on global Riemannian symmetric spaces and their heat kernels. 
In this section we assume that $G$ is a connected semisimple group and $G(\R)$ is 
of noncompact type.  Then $G(\R)$ is a semisimple real Lie group of noncompact 
type. Let
$K_\infty\subset G(\R)$ be a maximal compact subgroup and
\[
\widetilde X=G(\R)/K_\infty
\]
the associated Riemannian symmetric space. 
Let $\Gamma\subset G(\R)$ be a torsion free lattice and let 
$X=\Gamma\bs\widetilde X$. 
Let $\nu$ be a finite-dimensional unitary representation of $K_\infty$ on 
$(V_{\nu},\left<\cdot,\cdot\right>_{\nu})$. Let
\begin{align*}
\widetilde{E}_{\nu}:=G(\R)\times_{\nu}V_{\nu}
\end{align*}
be the associated homogeneous vector bundle over $\widetilde{X}$. Then 
$\left<\cdot,\cdot\right>_{\nu}$ induces a $G(\R)$-invariant metric 
$\tilde{h}_{\nu}$ on $\tilde{E}_{\nu}$. Let $\widetilde{\nabla}^{\nu}$ be the 
connection on $\tilde{E}_{\nu}$ induced by the canonical connection on the
principal $K_\infty$-fiber bundle $G(\R)\to G(\R)/K_\infty$. Then 
$\widetilde{\nabla}^{\nu}$ is $G(\R)$-invariant.
Let  
\begin{align*}
E_{\nu}:=\Gamma\bs \widetilde E_\nu
\end{align*}
be the associated locally homogeneous vector bundle over $X$. Since 
$\tilde{h}_{\nu}$ and $\widetilde{\nabla}^{\nu}$ are $G(\R)$-invariant, they push 
down to a metric $h_{\nu}$ and a connection $\nabla^{\nu}$ on $E_{\nu}$. Let
$C^\infty(\widetilde X,\widetilde E_\nu)$ resp. $C^\infty(X,E_\nu)$ denote the 
space of smooth sections of $\widetilde E_\nu$, resp. $E_\nu$. 
Let
\begin{align}\label{globsect}
\begin{split}
C^{\infty}(G(\R),\nu):=\{f:G(\R)\rightarrow V_{\nu}\colon f\in C^\infty,\:
f(gk)=&\nu(k^{-1})f(g),\\
&\forall g\in G(\R), \,\forall k\in K_\infty\},
\end{split}
\end{align}
Let $L^2(G(\R),\nu)$ be the corresponding $L^2$-space. There is a canonical
isomorphism
\begin{equation}\label{iso-glsect}
\widetilde A\colon C^\infty(\widetilde X,\widetilde E_\nu)\cong 
C^\infty(G(\R),\nu),
\end{equation}
(see \cite[p. 4]{Mia}). $\widetilde A$ extends to an isometry of the 
corresponding $L^2$-spaces. Let
\begin{align}\label{globsect1}
C^{\infty}(\Gamma\backslash G(\R),\nu):=\left\{f\in C^{\infty}(G(\R),\nu)\colon 
f(\gamma g)=f(g)\:\forall g\in G(\R), \forall \gamma\in\Gamma\right\}
\end{align}
and let $L^2(\Gamma\bs G(\R),\nu)$ be the corresponding $L^2$-space. The
isomorphism \eqref{iso-glsect} descends to isomorphisms
\begin{equation}\label{iso-glsect1}
 A\colon C^{\infty}(X,E_{\nu})\cong C^{\infty}(\Gamma\backslash G(\R),\nu),\quad
L^2(X,E_\nu)\cong L^2(\Gamma\bs G(\R),\nu).
\end{equation}
Let
$\widetilde{\Delta}_{\nu}={\widetilde{\nabla^\nu}}^{*}{\widetilde{\nabla}}^{\nu}$
be the Bochner-Laplace operator of $\widetilde{E}_{\nu}$. This is a 
$G(\R)$-invariant second order elliptic differential operator whose principal
symbol is given by
\[
\sigma_{\widetilde\Delta_\nu}(x,\xi)=\|\xi\|^2_x\cdot\Id_{E_{\nu,x}},\quad 
x\in\widetilde X, \;\xi\in T^\ast_x(\widetilde X).
\] 
Since 
$\widetilde{X}$ is complete, 
$\widetilde{\Delta}_{\nu}$ with domain the smooth compactly supported sections
is essentially self-adjoint \cite[p. 155]{LM}. Its self-adjoint extension
will be denoted by $\widetilde{\Delta}_\nu$ too. Let $\Omega\in\cZ(\gf_\C)$ 
and $\Omega_{K_\infty}\in\cZ(\kf)$ be the Casimir operators of $\gf$ and
$\kf$, respectively, where the latter is defined with respect to  the
restriction of the normalized Killing form of $\mathfrak{g}$ to $\mathfrak{k}$.
Then with respect to the isomorphism \eqref{iso-glsect} we have
\begin{align}\label{BLO}
\widetilde{\Delta}_{\nu}=-R(\Omega)+\nu(\Omega_{K_\infty}),
\end{align} 
where $R$ denotes the right regular representation of $G(\R)$ in 
$C^\infty(G(\R),\nu)$ (see \cite[Proposition 1.1]{Mia}). 

Let $e^{-t\widetilde\Delta_{\nu}}$, $t>0$, be the heat semigroup generated by 
$\widetilde\Delta_\nu$. It commutes with the action of $G(\R)$. With respect
to the isomorphism \eqref{iso-glsect} we may regard $e^{-t\widetilde\Delta_{\nu}}$ as a
bounded operator in $L^2(G(\R),\nu)$, which commutes with the action of $G(\R)$.
Hence it is a convolution operator, i.e., there exists a smooth map
\begin{equation}\label{heatker1}
H_t^\nu\colon G(\R)\to \End(V_\nu)
\end{equation}
such that
\[
(e^{-t\widetilde\Delta_{\nu}}\phi)(g)=\int_{G(\R)} H_t^\nu(g^{-1}g^\prime)(\phi(g^\prime))\;
dg^\prime,\quad \phi\in L^2(G(\R),\nu).
\]
The kernel $H_t^\nu$ satisfies
\begin{align}\label{propH}
{H}^{\nu}_{t}(k^{-1}gk')=\nu(k)^{-1}\circ {H}^{\nu}_{t}(g)\circ\nu(k'),
\:\forall k,k'\in K, \forall g\in G.
\end{align}
Moreover, proceeding as in the proof of \cite[Proposition 2.4]{BM} it follows 
that $H^\nu_t$ belongs to
$(\ccC^{q}(G(\R))\otimes\End(V_{\nu}))^{K_\infty\times K_\infty}$ for all $q>0$, where 
$\ccC^q(G(\R))$ is Harish-Chandra's Schwartz space of $L^q$-integrable 
rapidly decreasing functions on $G(\R)$. 

Let $\pi$ be a unitary representation of $G(\R)$ on a Hilbert space $\cH_\pi$.
Define a bounded operator on $\cH_{\pi}\otimes V_{\nu}$ by
\begin{align}\label{Definiton von pi(tilde(H))}
\tilde{\pi}(H^{\nu}_{t}(g)):=\int_{G(\R)}\pi(g)\otimes H^{\nu}_{t}(g)\;dg.
\end{align}
Then relative to the splitting 
\begin{align*}
\mathcal{H}_{\pi}\otimes V_{\nu}=\left(\mathcal{H}_{\pi}\otimes 
V_{\nu}\right)^{K_\infty}\oplus\left(\left(\mathcal{H}_{\pi}\otimes 
V_{\nu}\right)^{K_\infty}\right)^{\bot},
\end{align*}
$\tilde{\pi}(H^{\nu}_{t})$ has the form
\begin{align*}
\begin{pmatrix}
\pi({H}^{\nu}_{t})&0\\0&0
\end{pmatrix},
\end{align*}
where $\pi(H^{\nu}_{t})$ acts on $\left(\mathcal{H}_{\pi}\otimes 
V_{\nu}\right)^{K_\infty}$.  Assume that $\pi$ is irreducible. Let $\pi(\Omega)$
be the Casimir eigenvalue of $\pi$. Then as in \cite[Corollary 2.2]{BM} it 
follows from \eqref{BLO} that
\begin{equation}\label{integop}
\pi(H_t^\nu)=e^{t(\pi(\Omega)-\nu(\Omega_{K_\infty}))}\Id,
\end{equation}
where $\Id$ is the identity on $\left(\mathcal{H}_{\pi}\otimes 
V_{\nu}\right)^{K_\infty}$. Put
\begin{equation}\label{loctrace}
h_t^\nu(g):=\tr H_t^\nu(g),\quad g\in G(\R).
\end{equation}
Then $h_t^\nu\in\ccC^q(G(\R))$ for all $q>0$. 
Let $\pi$ be a unitary representation of $G(\R)$. Put
\[
\pi(h_t^\nu)=\int_{G(\R)}h_t^\nu(g)\pi(g)\; dg.
\]
Assume that $\pi(H_t^\nu)$ is a trace class operator. Then it follows as in 
\cite[Lemma 3.3]{BM} that $\pi(h_t^\nu)$ is a trace class operator and
\begin{equation}\label{equ-tr}
\Tr\pi(h_t^\nu)=\Tr\pi(H_t^\nu).
\end{equation}
Now assume that $\pi$ is a unitary admissible representation. Let 
$A:\mathcal{H}_\pi\rightarrow \mathcal{H}_\pi$ be a bounded operator which is an
intertwining operator for $\pi|_{K}$. Then $A\circ\pi(h_t^\nu)$ is again a
finite rank
operator. Define an operator $\tilde{A}$ on $\mathcal{H}_\pi\otimes V_\nu$ by
$\tilde{A}:=A\otimes\Id$.
Then by the same argument as in \cite[Lemma 5.1]{BM} one has
\begin{equation}\label{equtrace}
\Tr \left(\tilde{A}\circ
\tilde{\pi}(H_t^\nu)\right)=\Tr\left(A\circ\pi(h_t^\nu)\right).
\end{equation}
Together with \eqref{integop} we obtain
\begin{equation}\label{TrFT}
\Tr\left(A\circ\pi(h_t^\nu)\right)=e^{t(\pi(\Omega)-\nu(\Omega_{K_\infty}))}
\Tr\left(\tilde{A}|_{(\mathcal{H}_{\pi}\otimes V_{\nu})^{K}}\right).
\end{equation}

Next we consider the twisted Laplace operator. 
Let $\tau$ be an irreducible finite dimensional representation of $G(\R)$ on 
$V_{\tau}$. Let $F_{\tau}$ be the flat vector bundle over $X$ associated to 
the restriction of $\tau$ to $\Gamma$. Let $\widetilde E_\tau$ be the 
homogeneous vector bundle over $\widetilde X$ associated to $\tau|_{K_\infty}$ 
and let $E_\tau:=\Gamma\bs \widetilde E_\tau$. There is a canonical isomorphism
\begin{equation}\label{iso-vb}
E_\tau\cong F_\tau
\end{equation}
\cite[Proposition 3.1]{MM}. By \cite[Lemma 3.1]{MM}, there exists an 
inner product $\left<\cdot,\cdot\right>$ on $V_{\tau}$ such that 
\begin{enumerate}
\item $\left<\tau(Y)u,v\right>=-\left<u,\tau(Y)v\right>$ for all 
$Y\in\mathfrak{k}$, $u,v\in V_{\tau}$
\item $\left<\tau(Y)u,v\right>=\left<u,\tau(Y)v\right>$ for all 
$Y\in\mathfrak{p}$, $u,v\in V_{\tau}$.
\end{enumerate}
Such an inner product is called admissible. It is unique up to scaling. Fix an 
admissible inner product. Since $\tau|_{K_\infty}$ is unitary with respect to 
this inner product, it induces a metric on $E_{\tau}$, and by \eqref{iso-vb} 
on $F_\tau$, which we also call 
admissible. Let $\Lambda^{p}(F_{\tau})=\Lambda^pT^*(X)\otimes F_\tau$. 
By \eqref{iso-vb} $\Lambda^{p}(F_{\tau})$ is isomorphic to the locally
homogeneous vector bundle associated to the representation
\begin{align}\label{repr4}
\nu_{p}(\tau):=\Lambda^{p}\Ad^{*}\otimes\tau:\:K_\infty\rightarrow\GL
(\Lambda^{p}\mathfrak{p}^{*}\otimes V_{\tau}).
\end{align}
The space of smooth section of $\Lambda^{p}(F_{\tau})$ is the space 
$\Lambda^{p}(X,F_{\tau})$ of $F_\tau$-valued $p$-forms.
By  \eqref{globsect1} there is a canonical isomorphism
\begin{equation}\label{iso-sect}
\Lambda^{p}(X,F_{\tau})\cong C^{\infty}(\Gamma\backslash G,\nu_{p}(\tau)).
\end{equation}
Let $\Delta_p(\tau)$ be the Laplace operator in $\Lambda^{p}(X,F_{\tau})$.
Let $R_\Gamma$ be the right regular representation of $G(\R)$ in
$C^{\infty}(\Gamma\backslash G,\nu_{p}(\tau))$ and $\Omega$ the Casimir 
element of $G(\R)$.
By \cite{MM} it follows that with respect to the isomorphism \eqref{iso-sect}
we have
\begin{equation}\label{twist-lapl}
\Delta_p(\tau)=-R_\Gamma(\Omega)+\tau(\Omega).
\end{equation}
 Let $\widetilde \Delta_p(\tau)$ be the lift of 
$\Delta_p(\tau)$ to the universal covering $\widetilde X$. It acts in the 
space $\Lambda^p(\widetilde X,\widetilde F_\tau)$ of $p$-forms on $\widetilde X$
with values in the pull back $\widetilde F_\tau$ of $F_\tau$. Then by
\eqref{iso-glsect} we have
\[
\Lambda^p(\widetilde X,\widetilde F_\tau)\cong C^\infty(G(\R),\nu_p(\tau)).
\]
and with respect to this isomorphism we also have
\[
\widetilde \Delta_p(\tau)=-R(\Omega)+\tau(\Omega),
\]
where $R$ is the regular representation of $G(\R)$ in 
$C^\infty(G(\R),\nu_p(\tau))$. Using \eqref{BLO} we obtain
\begin{equation}
\widetilde \Delta_p(\tau)=\widetilde\Delta_{\nu_p(\tau)}+\tau(\Omega)-
\nu_p(\tau)(\Omega_{K_\infty}).
\end{equation}
We note that $\widetilde \Delta_p(\tau)$ is a 
formally self-adjoint, non-negative, elliptic second order differential 
operator. Regarded as operator in the Hilbert space $L^2\Lambda^p(X,F_\tau)$
of square integrable $F_\tau$-valued $p$-forms on $X$ with domain the space of
compactly supported smooth $p$-forms, 
it has a unique self-adjoint extension  which we also denote by 
$\widetilde \Delta_p(\tau)$. This is a non-negative self-adjoint 
operator in $\Lambda^p(\widetilde X,\widetilde F_\tau)$. Let 
$e^{-t\widetilde\Delta_p(\tau)}$, $t>0$, be the heat semigroup generated by 
$\widetilde \Delta_p(\tau)$. It is well known that $e^{-t\widetilde\Delta_p(\tau)}$
is an integral operator with a smooth kernel. Since $\widetilde \Delta_p(\tau)$
commutes with the action of $G(\R)$, $e^{-t\widetilde\Delta_p(\tau)}$ is a convolution
operator with kernel
\begin{equation}\label{heat-kern}
H^{\tau,p}_t\colon G(\R)\to \End(\Lambda^p\pg^\star\otimes V_\tau),
\end{equation}
which belongs to $C^\infty\cap L^2$, and satisfies the covariance property
\begin{equation}\label{covar}
H^{\tau,p}_t(k^{-1}gk')=\nu_p(\tau)(k)^{-1} H^{\tau,p}_t(g)\nu_p(\tau)(k')
\end{equation}
with respect to the representation \eqref{repr4}. Moreover, for all $q>0$ we 
have 
\begin{equation}\label{schwartz1}
H^{\tau,p}_t \in (\mathcal{C}^q(G(\R))\otimes
\End(\Lambda^p\pf^*\otimes V_\tau))^{K_\infty\times K_\infty}, 
\end{equation}
where $\mathcal{C}^q(G(\R))$ denotes Harish-Chandra's $L^q$-Schwartz space
(see \cite[Sect. 4]{MP2}). Let $h^{\tau,p}_t\in C^\infty(G(\R))$ be defined by
\begin{equation}\label{tr-kern}
h^{\tau,p}_t(g)=\tr H^{\tau,p}_t(g),\quad g\in G(\R). 
\end{equation}
Then $h_t^{\tau,p}\in\mathcal{C}^q(G(\R))$ for all $q>0$. 

\section{Analytic torsion}\label{sec-analtor}
\setcounter{equation}{0}

We briefly recall the definition of the analytic torsion. For details we
refer to \cite{MzM}. Let $G$ be a reductive algebraic group over $\Q$. Let
$K_\infty\subset G(\R)^1$  be a maximal compact subgroup and $\widetilde X=
G(\R)^1/K_\infty$. Let $K_f\subset G(\A_f)$ be an open compact subgroup. Let
\begin{equation}\label{adelic-quot2}
X(K_f):=G(\Q)\bs (\widetilde X\times G(\A_f))/K_f
\end{equation}
be the adelic quotient. It is the disjoint union of finitely many components
$\Gamma_i\bs \widetilde X$, where $\Gamma_i\subset G(\Q)$, $i=1,\ldots,m$, are
arithmetic subgroups. Let $\tau\in\Rep(G(\R)^1)$. Denote by $E_{\tau;i}$ the
locally flat vector bundle over $\Gamma_i\bs \widetilde X$, associated to
$\tau|_{\Gamma_i}$. Let $E_\tau$ be the disjoint union of the $E_{\tau;i}$. Then
$E_\tau$ is a flat vector bundle over $X(K_f)$. Let $\Delta_p(\tau)$ the 
Laplace operator on $E_\tau$-valued $p$-forms over $X(K_f)$. Let $h_t^{\tau,p}$
be the function defined by \eqref{tr-kern} and
let $\chi_{K_f}$ be the normalized characteristic function of $K_f$ in $G(\A_f)$
defined by \eqref{normal-chfct}. Put
\begin{equation}\label{testfct5}
\phi_t^{\tau,p}:=h_t^{\tau,p}\otimes \chi_{K_f}.
\end{equation}
Then  $\phi_t^{\tau,p}$ belongs to the adelic Schwartz space $\Co(G(\A)^1;K_f)$ 
(see section \ref{sec-prel}). Let $J_{\geo}(f)$, $f\in C_c^\infty(G(\A)^1)$ be 
the geometric side of the Arthur trace formula \cite{Ar1}. The distribution
$J_{\geo}$ extends to $\Co(G(\A)^1;K_f)$ (see \cite{FL1}). In 
\cite[(13.17)]{MzM} we defined the regularized trace of the heat operator
$e^{-t\Delta_p(\tau)}$ by
\begin{equation}\label{reg-trace7}
\Tr_{\reg}\left(e^{-t\Delta_p(\tau)}\right):=J_{\geo}(\phi_t^{\tau,p}).
\end{equation}
For the motivation for this definition we refer to \cite{MzM}. We only note
that if $X(K_f)$ is compact, then
$e^{-t\Delta_p(\tau)}$ is a trace class operator and the regularized trace is the
usual trace, which is equal to the spectral side of the trace formula. So in
this case, \eqref{reg-trace7} is just the trace formula. To define the
zeta function $\zeta_p(s,\tau)$ through the Mellin transform of the regularized
trace of the heat operator, we need to determine the asymptotic behavior of
$\Tr_{\reg}(e^{-t\Delta_p(\tau)})$ as $t\to\infty$ and $t\to 0$. This requires
additional assumptions. 

From now on we assume that $G=\GL(n)$ or $\SL(n)$. 
Let $\theta$ be the Cartan involution of $G(\R)^1$.
Let $\tau_\theta=\tau\circ\theta$. Assume that $\tau\neq\tau_\theta$. Then by
\cite[Proposition 13.4]{MzM} and the trace formula we have
\begin{equation}\label{large-time7}
\Tr_{\reg}\left(e^{-t\Delta_p(\tau)}\right)=O(e^{-ct})
\end{equation}
as $t\to\infty$. The existence of an asymptotic expansion as $t\to0$ follows
from \cite[Theorem 1.1]{MzM}. Assume that $K_f$ is contained in $K(N)$ for 
some $N\ge 3$. Then there is an asymptotic expansion
\begin{equation}\label{small-time5}
\Tr_{\reg}\left(e^{-t\Delta_p(\tau)}\right)\sim t^{-d/2}\sum_{j=0}^\infty a_jt^j+
t^{-(d-1)/2}\sum_{j=0}^\infty\sum_{i=0}^{r_j}b_{ij}t^{j/2}(\log t)^i
\end{equation}
as $t\to 0$. Thus, under the assumptions above, the integral
\begin{equation}\label{anal-tor8}
\zeta_p(s,\tau):=\frac{1}{\Gamma(s)}\int_0^\infty 
\Tr_{\reg}\left(e^{-t\Delta_p(\tau)}\right)t^{s-1} dt
\end{equation}
converges absolutely and uniformly on compact subsets of the half-plane
$\Re(s)>d/2$, and admits a meromorphic extension to the entire complex plane.
Due to the logarithmic terms in the expansion \eqref{small-time5}, the zeta
function $\zeta_p(s,\tau)$ may have a pole at $s=0$. The analytic torsion
is then defined by \eqref{analtor}. 

In the case of $G=\GL(3)$ we are able to determine the coefficients of the 
log-terms. This shows that the 
zeta functions definitely have a pole at $s=0$. However, the combination
$\sum_{p=1}^5 (-1)^p p \zeta_p(s;\tau)$ turns out to be holomorphic at $s=0$ 
(see \cite[sect. 14]{MzM}) and we can define the logarithm of the analytic 
torsion by
\[
\log T_{X(K_f)}(\tau)=\frac{d}{ds}
\left(\frac{1}{2}\sum_{p=1}^5 (-1)^p p \zeta_p(s;\tau)\right)\bigg|_{s=0}.
\]

\section{Review of the spectral side of the trace formula}
\label{sec-specside}
\setcounter{equation}{0}

In this section $G$ is an arbitrary reductive algebraic group over $\Q$.
Arthur's (non-invariant) trace formula is the equality 
\begin{equation}\label{tracef1}
J_{\geo}(f)=J_{\spec}(f),\quad f\in C_c^\infty(G(\A)^1),
\end{equation}
of two distributions on $G(\A)^1$, namely the equality of
the geometric side $J_{\geo}(f)$ and the spectral side 
$J_{\spec}(f)$ of the trace formula.
In this section we recall the definition of the spectral side, and in particular
the refinement of the spectral expansion obtained in \cite{FLM1}, which we need
for our purpose. Combining
\cite{FLM1} and \cite{FL1}, it follows that \eqref{tracef1} extends continuously
to $f\in\Co(G(\A)^1)$.

The main ingredient of the spectral side  are logarithmic derivatives of 
intertwining operators. We briefly recall the structure of the intertwining 
operators.

Let $P\in\cP(M)$. 
Let $U_P$ be the unipotent radical of $P$. 
Recall that we denote  by $\rts_P\subset\af_P^*$ the set of reduced roots of 
$A_M$ of the Lie algebra $\mathfrak{u}_P$ of $U_P$.
Let $\srts_P$ be the subset of simple roots of $P$, which is a basis for 
$(\af_P^G)^*$.
Write $\af_{P,+}^*$ for the closure of the Weyl chamber of $P$, i.e.
\[
\aaa_{P,+}^*=\{\lambda\in\aaa_M^*:\sprod{\lambda}{\alpha^\vee}\ge0
\text{ for all }\alpha\in\rts_P\}
=\{\lambda\in\aaa_M^*:\sprod{\lambda}{\alpha^\vee}\ge0\text{ for all }
\alpha\in\srts_P\}.
\]
Denote by $\modulus_P$ the modulus function of $P(\A)$.
Let $\bar\AF^2(P)$ be the Hilbert space completion of
\[
\{\phi\in C^\infty(M(\Q)U_P(\A)\bs G(\A)):\modulus_P^{-\frac12}\phi(\cdot x)\in
L^2_{\disc}(\Ai M(\Q)\bs M(\A)),\ \forall x\in G(\A)\}
\]
with respect to the inner product
\[
(\phi_1,\phi_2)=\int_{\Ai M(\Q)\bU_P(\A)\bs \G(\A)}\phi_1(g)
\overline{\phi_2(g)}\ dg.
\]
Let $\alpha\in\rts_M$.
We say that two parabolic subgroups $P,Q\in\cP(M)$ are \emph{adjacent} along 
$\alpha$, and write $P|^\alpha Q$, if $\rts_P\cap-\rts_Q=\{\alpha\}$.
Alternatively, $P$ and $Q$ are adjacent if the group $\langle P,Q\rangle$
generated by $P$ and $Q$ belongs to $\cF_1(M)$ (see \eqref{f1} for its
definition).
Any $R\in\cF_1(\M)$ is of the form $\langle P,Q\rangle$, where $P,Q$ are
the elements of $\cP(M)$ contained in $R$. We have $P|^\alpha Q$ with 
$\alpha^\vee\in\rts_P^\vee \cap\af^R_M$.
Interchanging $P$ and $Q$ changes $\alpha$ to $-\alpha$.

For any $P\in\cP(M)$ let $\Ht_P\colon G(\A)\rightarrow\af_P$ be the 
extension of $\Ht_M$ to a left $U_P(\A)$-and right $\K$-invariant map.
Denote by $\cA^2(P)$ the dense subspace of $\bar\cA^2(P)$ consisting of its 
$\K$- and $\zzz$-finite vectors,
where $\zzz$ is the center of the universal enveloping algebra of 
$\mathfrak{g} \otimes \C$.
That is, $\cA^2(P)$ is the space of automorphic forms $\phi$ on 
$U_P(\A)M(\Q)\bs G(\A)$ such that
$\modulus_P^{-\frac12}\phi(\cdot k)$ is a square-integrable automorphic form on
$\Ai M(\Q)\bs M(\A)$ for all $k\in\K$.
Let $\rho(P,\lambda)$, $\lambda\in\af_{M,\C}^*$, be the induced
representation of $G(\A)$ on $\bar\cA^2(P)$ given by
\[
(\rho(P,\lambda,y)\phi)(x)=\phi(xy)e^{\sprod{\lambda}{\Ht_P(xy)-\Ht_P(x)}}.
\]
It is isomorphic to the induced representation 
\[
\Ind_{P(\A)}^{G(\A)}\left(L^2_{\disc}(\Ai M(\Q)\bs M(\A))
\otimes e^{\sprod{\lambda}{\Ht_M(\cdot)}}\right).
\]

For $P,Q\in\cP(M)$ let
\[
M_{Q|P}(\lambda):\cA^2(P)\to\cA^2(Q),\quad\lambda\in\af_{M,\C}^*,
\]
be the standard \emph{intertwining operator} \cite[\S 1]{Ar9}, which is the 
meromorphic continuation in $\lambda$ of the integral
\[
[M_{Q|P}(\lambda)\phi](x)=\int_{U_Q(\A)\cap U_P(\A)\bs U_Q(\A)}\phi(nx)
e^{\sprod{\lambda}{\Ht_P(nx)-\Ht_Q(x)}}\ dn, \quad \phi\in\cA^2(P), \ x\in G(\A).
\]
Given $\pi\in\Pi_{\di}(M(\A))$, let $\cA^2_\pi(P)$ be the space of all 
$\phi\in\cA^2(P)$ for which the function 
$M(\A)\ni x\mapsto \delta_P^{-\frac{1}{2}}\phi(xg)$,
$g\in G(\A)$, belongs to the $\pi$-isotypic subspace of the space
$L^2(\Ai M(\Q)\bs M(\A))$.
For any $P\in\cP(\M)$ we have a canonical isomorphism of 
$G(\A_f)\times(\LieG_{\C},K_\infty)$-modules
\[
j_P:\Hom(\pi,L^2(\Ai M(\Q)\bs M(\A)))\otimes 
\Ind_{P(\A)}^{G(\A)}(\pi)\rightarrow\cA^2_\pi(P).
\]
If we fix a unitary structure on $\pi$ and endow 
$\Hom(\pi,L^2(\Ai M(\Q)\bs M(\A)))$ with the inner product 
$(A,B)=B^\ast A$
(which is a scalar operator on the space of $\pi$), the isomorphism $j_P$ 
becomes an isometry. 

Suppose that $P|^\alpha Q$.
The operator $M_{Q|P}(\pi,s):=M_{Q|P}(s\varpi)|_{\cA^2_\pi(P)}$, where $\varpi\in
\af^\ast_M$ is such that $\langle\varpi,\alpha^\vee\rangle=1$, admits a 
normalization by a global factor
$n_\alpha(\pi,s)$ which is a meromorphic function in $s$. We may write
\begin{equation} \label{normalization}
M_{Q|P}(\pi,s)\circ j_P=n_\alpha(\pi,s)\cdot j_Q\circ(\Id\otimes R_{Q|P}(\pi,s))
\end{equation}
where $R_{Q|P}(\pi,s)=\otimes_v R_{Q|P}(\pi_v,s)$ is the product
of the locally defined normalized intertwining operators and 
$\pi=\otimes_v\pi_v$
\cite[\S 6]{Ar9}, (cf.~\cite[(2.17)]{Mu2}). In many cases, the 
normalizing factors can be expressed in terms automorphic $L$-functions 
\cite{Sha1}, \cite{Sha2}. For example, let $G=\GL(n)$. Then
the global normalizing factors $n_\alpha$ can be expressed in terms
of Rankin-Selberg $L$-functions. The
known properties of these functions are collected and analyzed in 
\cite[\S\S 4,5]{Mu1}.
Write $M \simeq \prod_{i=1}^r \GL (n_i)$, where the root $\alpha$ is trivial on 
$\prod_{i \ge 3} \GL (n_i)$,
and let $\pi \simeq \otimes \pi_i$ with representations $\pi_i \in 
\Pi_{\disc}(\GL(n_i,\A))$.
Let $L(s,\pi_1\times\tilde\pi_2)$ be the completed Rankin-Selberg $L$-function 
associated to $\pi_1$ and $\pi_2$. It satisfies the functional equation
\begin{equation}\label{functequ}
L(s,\pi_1\times\tilde\pi_2)=\eps(\frac12,\pi_1\times\tilde\pi_2)
N(\pi_1\times\tilde\pi_2)^{\frac12-s}L(1-s,\tilde\pi_1\times\pi_2)
\end{equation}
where $\abs{\eps(\frac12,\pi_1\times\tilde\pi_2)}=1$ and $N(\pi_1\times\tilde
\pi_2)\in\N$ is the conductor. Then we have
\begin{equation}\label{rankin-selb}
n_\alpha (\pi,s) = \frac{L(s,\pi_1\times\tilde\pi_2)}{\eps(\frac12,\pi_1\times
\tilde\pi_2)N(\pi_1\times\tilde\pi_2)^{\frac12-s}L(s+1,\pi_1\times\tilde\pi_2)}.
\end{equation}

We now turn to the spectral side. Let $L \supset M$ be Levi subgroups in 
$\levis$, $P \in \PPP (M)$, and
let $m=\dim\aaa_L^G$ be the co-rank of $L$ in $G$.
Denote by $\bases_{P,L}$ the set of $m$-tuples $\bss=(\beta_1^\vee,\dots,
\beta_m^\vee)$
of elements of $\rts_P^\vee$ whose projections to $\af_L$ form a basis for 
$\af_L^G$.
For any $\bss=(\beta_1^\vee,\dots,\beta_m^\vee)\in\bases_{P,L}$ let
$\vol(\bss)$ be the co-volume in $\af_L^G$ of the lattice spanned by $\bss$ 
and let
\begin{align*}
\Xi_L(\bss)&=\{(Q_1,\dots,Q_m)\in\cF_1(M)^m: \ \ \beta_i^{\vee}\in\af_M^{Q_i}, 
\, i = 1, \dots, m\}\\&=
\{\langle P_1,P_1'\rangle,\dots,\langle P_m,P_m'\rangle): \ \ 
P_i|^{\beta_i}P_i', \, 
i = 1, \dots, m\}.
\end{align*}

For any smooth function $f$ on $\af_M^*$ and $\mu\in\af_M^*$ denote by 
$D_\mu f$ the directional derivative of $f$ along $\mu\in\af_M^*$.
For a pair $P_1|^\alpha P_2$ of adjacent parabolic subgroups in $\cP(M)$ write
\begin{equation}\label{intertw2}
\delta_{P_1|P_2}(\lambda)=M_{P_2|P_1}(\lambda)D_\varpi M_{P_1|P_2}(\lambda):
\cA^2(P_2)\rightarrow\cA^2(P_2),
\end{equation}
where $\varpi\in\af_M^*$ is such that $\sprod{\varpi}{\alpha^\vee}=1$.
\footnote{Note that this definition differs slightly from the definition of
$\delta_{P_1|P_2}$ in \cite{FLM1}.} Equivalently, writing
$M_{P_1|P_2}(\lambda)=\Phi(\sprod{\lambda}{\alpha^\vee})$ for a
meromorphic function $\Phi$ of a single complex variable, we have
\[
\delta_{P_1|P_2}(\lambda)=\Phi(\sprod{\lambda}{\alpha^\vee})^{-1}
\Phi'(\sprod{\lambda}{\alpha^\vee}).
\]
For any $m$-tuple $\dtup=(Q_1,\dots,Q_m)\in\Xi_L(\bss)$
with $Q_i=\langle P_i,P_i'\rangle$, $P_i|^{\beta_i}P_i'$, denote by 
$\Delta_{\dtup}(P,\lambda)$
the expression
\begin{equation}\label{intertw3}
\frac{\vol(\bss)}{m!}M_{P_1'|P}(\lambda)^{-1}\delta_{P_1|P_1'}(\lambda)M_{P_1'|P_2'}(\lambda) \cdots
\delta_{P_{m-1}|P_{m-1}'}(\lambda)M_{P_{m-1}'|P_m'}(\lambda)\delta_{P_m|P_m'}(\lambda)M_{P_m'|P}(\lambda).
\end{equation}
In \cite[pp. 179-180]{FLM1} we defined a (purely combinatorial) map $\dtup_L: \bases_{P,L} \to \FFF_1 (M)^m$ with the property that
$\dtup_L(\bss) \in \Xi_L (\bss)$ for all $\bss \in \bases_{P,L}$.\footnote{The map $\dtup_L$ depends in fact on the additional choice of
a vector $\underline{\mu} \in (\aaa^*_M)^m$ which does not lie in an explicit finite
set of hyperplanes. For our purposes, the precise definition of $\dtup_L$ is immaterial.}

For any $s\in W(M)$ let $L_s$ be the smallest Levi subgroup in $\levis(M)$
containing $w_s$. We recall that $\aaa_{L_s}=\{H\in\aaa_M\mid sH=H\}$.
Set
\[
\iota_s=\abs{\det(s-1)_{\aaa^{L_s}_M}}^{-1}.
\]
For $P\in\FFF(M_0)$ and $s\in W(M_P)$ let
$M(P,s):\AF^2(P)\to\AF^2(P)$ be as in \cite[p.~1309]{Ar3}.
$M(P,s)$ is a unitary operator which commutes with the operators $\rho(P,\lambda,h)$ for $\lambda\in\iii\aaa_{L_s}^*$.
Finally, we can state the refined spectral expansion.

\begin{theo}[\cite{FLM1}] \label{thm-specexpand}
For any $h\in C_c^\infty(G(\A)^1)$ the spectral side of Arthur's trace formula is given by
\begin{equation}\label{specside1}
J_{\spec}(h) = \sum_{[M]} J_{\spec,M} (h),
\end{equation}
$M$ ranging over the conjugacy classes of Levi subgroups of $G$ (represented by members of $\mathcal{L}$),
where
\begin{equation}\label{specside2}
J_{\spec,M} (h) =
\frac1{\card{W(M)}}\sum_{s\in W(M)}\iota_s
\sum_{\bss\in\bases_{P,L_s}}\int_{\iii(\aaa^G_{L_s})^*}
\tr(\Delta_{\dtup_{L_s}(\bss)}(P,\lambda)M(P,s)\rho(P,\lambda,h))\ d\lambda
\end{equation}
with $P \in \PPP(M)$ arbitrary.
The operators are of trace class and the integrals are absolutely convergent
with respect to the trace norm and define distributions on $\Co(G(\A)^1)$.
\end{theo}
Note that the term corresponding to $M=G$ is $J_{\spec,G} (h) = \tr R_{\disc}(h)$. 
Next assume that $M$ is the Levi subgroup of a maximal parabolic subgroup $P$.
Furthermore, let $L=M$. Let $\bar P$ be the opposite parabolic subgroup to $P$. 
Then up to a constant, the contribution to the spectral side is given by
\[
\sum_{\pi\in\Pi_{\di}(M(\A)^1)}\int_{i\af^\ast}\tr(M_{\bar P|P}(\pi,\lambda)^{-1}
\frac{d}{dz}M_{\bar P|P}(\pi,\lambda)M(P,s)\rho(P,\pi,\lambda,h))\;d\lambda.
\]

\section{Large time behavior of the regularized trace}
\label{sec-large-time}
\setcounter{equation}{0}

The purpose of this section is to improve \eqref{large-time7} so that the
estimations are uniform with respect to $K_f$. To this end we use the trace 
formula \eqref{tracef1}. 
By Theorem \ref{thm-specexpand}, $J_{\spec}$  is a distribution on 
$\Co(G(\A);K_f)$ and by \cite[Theorem 7.1]{FL1}, $J_{\geo}$ is continuous on
$\Co(G(\A);K_f)$. This implies  that
\eqref{tracef1} holds for $\phi_t^{\tau,p}$. Using the definition
\eqref{reg-trace7} of the regularized trace and the trace formula we get
\begin{equation}\label{regtrace1a}
\Tr_{\reg}\left(e^{-t\Delta_p(\tau)}\right)=J_{\spec}(\phi_t^{\tau,p}).
\end{equation}
Now we apply Theorem \ref{thm-specexpand} to study the asymptotic behavior
as $t\to\infty$ of the right hand side. Let $M\in\cL$ and $P\in\cP(M)$.
Recall that 
$L^2_{\di}(A_M(\R)^0 M(\Q)\bs M(\A))$ splits as the completed direct sum of its
$\pi$-isotypic components for $\pi\in\Pi_{\di}(M(\A))$. We have a corresponding
decomposition of $\bar{\cA}^2(P)$ as a direct sum of Hilbert spaces
$\hat\oplus_{\pi\in\Pi_{\di}(M(\A))}\bar{\cA}^2_\pi(P)$. Similarly, we have the
algebraic direct sum decomposition
\[
\cA^2(P)=\bigoplus_{\pi\in\Pi_{\di}(M(\A))}\cA^2_\pi(P),
\]
where $\cA^2_\pi(P)$ is the ${\bf K}$-finite part of $\bar{\cA}^2_\pi(P)$. For 
$\sigma\in\widehat{K_\infty}$ let $\cA^2_\pi(P)^\sigma$ be the $\sigma$-isotypic
subspace. Then $\cA^2_\pi(P)$ decomposes as
\[
\cA^2_\pi(P)=\bigoplus_{\sigma\in\widehat{K_\infty}}\cA^2_\pi(P)^\sigma.
\]
Let $\cA^2_\pi(P)^{K_f}$ be the subspace of $K_f$-invariant functions in
$\cA^2_\pi(P)$, and for any $\sigma\in\widehat{K_\infty}$ let 
$\cA^2_\pi(P)^{K_f,\sigma}$ be the $\sigma$-isotypic subspace of 
$\cA^2_\pi(P)^{K_f}$. Recall that $\cA^2_\pi(P)^{K_f,\sigma}$ is finite dimensional.
Let $M_{Q|P}(\pi,\lambda)$ denote the restriction of $M_{Q|P}(\lambda)$ to
$\cA^2_\pi(P)$. Recall that the operator $\Delta_\chi(P,\lambda)$, which 
appears in the formula \eqref{specside2}, is defined by \eqref{intertw3}.
Its definition involves the intertwining operators $M_{Q|P}(\lambda)$. If we 
replace $M_{Q|P}(\lambda)$ by its restriction $M_{Q|P}(\pi,\lambda)$ to
$\cA^2_\pi(P)$, we obtain the restriction $\Delta_\chi(P,\pi,\lambda)$ of
$\Delta_\chi(P,\lambda)$ to $\cA^2_\pi(P)$. Similarly, let $\rho_\pi(P,\lambda)$
be the induced representation in $\bar{\cA}^2_\pi(P)$. 
Fix $\beta\in\bases_{P,L_s}$ and $s\in W(M)$.  Then for the integral 
on the right of \eqref{specside2} with $h=\phi_t^{\tau,p}$ we get
\begin{equation}\label{specside3}
\sum_{\pi\in\Pi_{\di}(M(\A))}\int_{i(\af^G_{L_s})^*}\Tr\left(
\Delta_{\dtup_{L_s}(\bss)}(P,\pi,\lambda)M(P,\pi,s)\rho_\pi(P,\lambda,\phi^{\tau,p}_t)
\right)\;d\lambda.
\end{equation}
Let $P,Q\in\cP(M)$ and $\nu\in\Pi(K_\infty)$.  Denote by 
$\widetilde M_{Q|P}(\pi,\nu,\lambda)$ the restriction of 
\[
M_{Q|P}(\pi,\lambda)\otimes\Id\colon \cA^2_\pi(P)\otimes V_{\nu}\to 
\cA^2_\pi(P)\otimes V_{\nu}
\]
to $(\cA^2_\pi(P)^{K_f}\otimes V_{\nu})^{K_\infty}$. Denote by 
$\widetilde\Delta_{\dtup_{L_s}(\bss)}(P,\pi,\nu,\lambda)$ and 
$\widetilde M(P,\pi,\nu,s)$ the corresponding restrictions. 
Let $m(\pi)$ denote the multiplicity with which $\pi$ occurs in the regular
representation of $M(\A)$ in $L^2_{\di}(M(\Q)\bs M(\A))$. Then
\begin{equation}\label{ind-repr}
\rho_\pi(P,\lambda)\cong \oplus_{i=1}^{m(\pi)}\Ind_{P(\A)}^{G(\A)}(\pi,\lambda).
\end{equation}
Let $\pi=\pi_\infty\otimes\pi_f$, where $\pi_\infty$ and $\pi_f$ are irreducible
unitary representations of $M(\R)$ and $M(\A_f)$, respectively. Then
\[
\Ind_{P(\A)}^{G(\A)}(\pi,\lambda)= \Ind_{P(\R)}^{G(\R)}(\pi_\infty,\lambda)\otimes
\Ind_{P(\A_f)}^{G(\A_f)}(\pi_f,\lambda).
\]
Let $\H(\pi_\infty)$ and $\H(\pi_f)$ denote the Hilbert space of $\pi_\infty$ and
$\pi_f$, respectively. Let $\omega(\pi_\infty,\lambda)$ be 
 the Casimir eigenvalue of the induced representation 
$\Ind_{P(\R)}^{G(\R)}(\pi_\infty,\lambda)$ and let $\Pi_{K_f}$ be the orthogonal
projection of $\H(\pi_f)$ onto the subspace $\H(\pi_f)^{K_f}$ of 
$K_f$-invariant vectors. Then by \eqref{twist-lapl} it follows 
that 
\[
\Ind_{P(\R)}^{G(\R)}(\pi_\infty,\lambda,h_t^{\tau,p})=e^{t(\tau(\Omega)-
\omega(\pi_\infty,\lambda))}\Id,
\]
where $\Id$ is the identity on $(\H(\pi_\infty)\otimes\Lambda^p\pg^\star\otimes
V_\tau)^{K_\infty}$.  Furthermore,
\[
\Ind_{P(\A_f)}^{G(\A_f)}(\pi_f,\lambda,\chi_{K_f})=\Pi_{K_f}.
\]
Let $\Pi_{K_f,\nu_p(\tau)}$ denote the orthogonal projection onto 
$\bar\cA_\pi^2(P)^{K_f,\nu_p(\tau)}$. Then it follows that
\begin{equation}\label{repr1}
\rho_\pi(P,\lambda,\phi_t^{\tau,p})=e^{t(\tau(\Omega)-
\omega(\pi_\infty,\lambda))}\Pi_{K_f,\nu_p(\tau)}.
\end{equation}
Fix positive restricted roots of $\af_P$ and let $\rho_{\af_P}$ denote the 
corresponding half-sum of these roots. For $\xi\in \Pi(M(\R))$ and
$\lambda\in\af^\ast_P$ let 
\[
\pi_{\xi,\lambda}:=\Ind_{P(\R)}^{G(\R)}(\xi\otimes e^{i\lambda})
\]
be the unitary induced representation. Let $\xi(\Omega_M)$ be the Casimir
eigenvalue of $\xi$. Define a constant $c(\xi)$ by
\begin{equation}\label{casimir4}
c(\xi):=-\langle\rho_{\af_P},\rho_{\af_P}\rangle+\xi(\Omega_M).
\end{equation}
Then for $\lambda\in\af^\ast_P$ one has
\begin{equation}\label{casimir5}
\pi_{\xi,\lambda}(\Omega)=-\|\lambda\|^2+c(\xi)
\end{equation}
(see \cite[Theorem 8.22]{Kn}). Let 
\begin{equation}\label{def-F}
\cT:=\{\nu\in\Pi(K_\infty)\colon [\nu_p(\tau)\colon\nu]\neq 0\}.
\end{equation}
Using \eqref{repr1} and \eqref{TrFT}, it follows that \eqref{specside3} 
is equal to
\begin{equation}\label{specside4}
\begin{split}
\sum_{\pi\in\Pi_{\di}(M(\A))}\sum_{\nu\in\cT}&e^{-t(\tau(\Omega)-c(\pi_\infty))}\\
&\int_{i(\af^G_{L_s})^*}e^{-t\|\lambda\|^2}\Tr\left(
\widetilde\Delta_{\dtup_{L_s}(\bss)}(P,\pi,\nu,\lambda)
\widetilde M(P,\pi,\nu,s)\right)\;d\lambda.
\end{split}
\end{equation}
 
Using that $M(P,\pi,s)$ is unitary, it follows that \eqref{specside4} can be
estimated by
\begin{equation}\label{est-specside}
\begin{split}
\sum_{\pi\in\Pi_{\di}(M(\A))}
\sum_{\nu\in\cT}&
\dim\left(\cA^2_\pi(P)^{K_f,\nu}\right)\\
&\cdot e^{-t(\tau(\Omega)-c(\pi_\infty))}\int_{i(\af^G_{L_s})^*}e^{-t\|\lambda\|^2}\|
\widetilde\Delta_{\dtup_{L_s}(\bss)}(P,\pi,\nu,\lambda)\|\;d\lambda.
\end{split}
\end{equation}
First we estimate the integral in \eqref{est-specside}.
Let $\bss=(\beta_1^\vee,\dots,\beta_m^\vee)$ and $\dtup_{L_s}(\bss)=
(Q_1,\dots,Q_m)\in\Xi_{L_s}(\bss)$ with with $Q_i=\langle P_i,P_i'\rangle$, 
$P_i|^{\beta_i}P_i'$, $i=1,\dots,m$.
 Using the definition \eqref{intertw3} of 
$\Delta_{\dtup_{L_s}(\bss)}(P,\pi,\nu,\lambda)$, it follows that we can bound the
integral by a constant multiple of
\begin{equation}\label{est-integral}
\dim(\nu)\int_{i(\af^G_{L_s})^*}e^{-t\|\lambda\|^2}\prod_{i=1}^m\left\| 
\delta_{P_i|P_i^\prime}(\lambda)\Big|_{\cA^2_\pi(P_i^\prime)^{K_f,\nu}}\right\|
\;d\lambda.
\end{equation}
We introduce new coordinates $s_i:=\langle\lambda,\beta_i^\vee\rangle$,
$i=1,\dots,m$, on $(\af^G_{L_s,\C})^\ast$. Using \eqref{normalization}, we
can write
\begin{equation}\label{delta}
\delta_{P_i|P_i^\prime}(\lambda)=\frac{n^\prime_{\beta_i}(\pi,s_i)}{n_{\beta_i}(\pi,s_i)}
+j_{P_i^\prime}\circ(\Id\otimes R_{P_i|P_i^\prime}(\pi,s_i)^{-1}
R^\prime_{P_i|P_i^\prime}(\pi,s_i))\circ j_{P_i^\prime}^{-1}.
\end{equation}
In \cite[Definition 5.2, Definition 5.9]{FLM2} two conditions, called (TWN)
and (BD), for an arbitrary reductive group have been formulated, which imply 
appropriate estimations for the
terms on the right. Furthermore, in \cite[Prop. 5.5, Prop. 5.15]{FLM2} it
was shown that the conditions (TWN) and (BD) both hold for $\GL(n)$ and
$\SL(n)$. Assume that the conditions (TWN) and (BD) hold for $G$. Then
as in \cite[(22)]{FLM2} this implies that for any $\epsilon>0$
and sufficiently large $k$ and $m$ one has
\begin{equation}\label{estim3}
\begin{split}
\int_{i(\af^G_{L_s})^*}(1+\|\lambda\|)^{-k}\prod_{i=1}^m\left\| 
\delta_{P_i|P_i^\prime}(\lambda)\Big|_{\cA^2_\pi(P_i^\prime)^{K_f,\nu}}\right\|
\;d\lambda
\ll_{\varepsilon,\cT}\Lambda_M(\pi_\infty;G_M)^m\level(K_f;G_M^+)^\varepsilon.
\end{split}
\end{equation}
for all $\nu\in\cT$. To estimate $\Lambda_M(\pi_\infty;G_M)$ we first recall 
Vogan's definition of a norm on $\|\cdot\|$ on $\Pi(K_\infty)$.
Let $\chi_\mu$ be the highest weight of an arbitrary irreducible constituent
of $\mu|_{K_\infty^0}$ with respect to a maximal torus of $K_\infty^0$ and the 
choice of a system of positive roots. Let $\rho$ be the half sum of all
positive roots with multiplicities. For $\mu\in\Pi(K_\infty)$ the norm $\|\mu\|$
is defined by $\|\mu\|=\|\chi_\mu+2\rho\|^2$. A minimal $K_\infty$-type of a 
representation of $G(\R)$ is then a $K_\infty$-type minimizing $\|\cdot\|$. 
For $\pi\in \Pi(M(\A))$ denote by $\lambda_{\pi_\infty}$ the Casimir 
eigenvalue of the restriction of $\pi_\infty$ to $M(\R)^1$. Let 
\begin{equation}\label{lowest-k-type}
\Lambda_\pi=\min_{\tau}\sqrt{\lambda_{\pi_\infty}^2+\lambda_\tau^2},
\end{equation}
where $\tau$ runs over the lowest $K_\infty$-types of the induced representation
$\Ind_{P(\R)}^{G(\R)}(\pi_\infty)$. Then by \cite[(10)]{FLM2} we have
\begin{equation}
1\le\Lambda_M(\pi;G_M)\le 1+\Lambda_\pi^2.
\end{equation}
Now observe that $\dim\cA^2_\pi(P)^{K_f,\nu}=0$, unless  
$[\Ind_{P(\R)}^{G(\R)}(\pi_\infty)|_{K_\infty}\colon \nu]>0$. Thus for a minimal 
$K_\infty$-type $\tau$ of $\Ind_{P(\R)}^{G(\R)}(\pi_\infty)$ one has $\lambda_\tau^2
\le\lambda_\nu^2$. Since $\cT$ is finite, there exists $C>0$ such that
\begin{equation}\label{minimal}
\Lambda_\pi\le C(1+|\lambda_{\pi_\infty}|)
\end{equation}
for all $\pi\in\Pi_{\di}(M(\A))$ with $\dim\cA^2_\pi(P)^{K_f,\nu}\neq 0$.
Thus it follows that for $t\ge 1$, \eqref{est-specside} can be estimated by a 
constant
times
\begin{equation}\label{est-specside1}
\sum_{\pi\in\Pi_{\di}(M(\A))}
\sum_{\nu\in\cT}\dim\left(\cA^2_\pi(P)^{K_f,\nu}\right)e^{-t(\tau(\Omega)-c(\pi_\infty))}
(1+|\lambda_{\pi_\infty}|)^m\level(K_f;G_M^+)^\varepsilon.
\end{equation}
To continue with the estimation, we need the following lemma.

\begin{lem}\label{est-casim}
Let $P=MAN$ be a parabolic subgroup of $G$ and let 
$K_{M,\infty}=M(\R)\cap K_\infty$. Let $(\tau,V_\tau)\in\Rep(G(\R))$. Assume that 
$\tau\ncong\tau_\theta$. There exists $\delta>0$ such that for all
$(\xi,W_\xi)\in\Pi(M(\R)^1)$ satisfying
$\dim(W_\xi\otimes\Lambda^p\pg^\ast\otimes V_\tau)^{K_{M,\infty}}\neq 0$
one has
\[
\tau(\Omega)-c(\xi)\ge \delta.
\]
\end{lem}
\begin{proof}
First consider the case $P=G$. 
In the proof of Lemma 4.1 in \cite{BV} it is shown that there exists $\delta>0$
such that 
\begin{equation}\label{est-casimir1}
\tau(\Omega)-\pi(\Omega)\ge\delta
\end{equation}
for each irreducible unitary representation $\pi$ of $G(\R)$ for which
\[
\Hom_{K_\infty}(\Lambda^p\pg\otimes V_\tau,\pi)\neq 0.
\]
In fact, the proof goes through for every unitary representation $\pi$ of
$G(\R)$ such that $\pi(\Omega)$ is a scalar (see \cite[\S II, Prop. 6.12]{BW}. 

Now let $P=MAN$ be a proper parabolic subgroup of $G$. 
Let $\xi\in\Pi(M(\R)^1)$ with $\dim(W_{\xi}\otimes\Lambda^p\pg^\ast
\otimes V_\tau)^{K_{M,\infty}}\neq 0$ and $\lambda\in\af^\star$. 
Consider the induced representation $\pi_{\xi,\lambda}$. 
By Frobenius reciprocity and the assumption on $\xi$ we have
\[
\dim\left(W_\xi\otimes\Lambda^p\pg^\ast\otimes V_\tau\right)^{K_{M,\infty}}=
\dim\left(\H_{\xi,\lambda}\otimes\Lambda^p\pg^\ast\otimes V_\tau\right)^{K_{\infty}}
\neq 0.
\]
Recall that $\pi_{\xi,\lambda}(\Omega)$ is a scalar given by \eqref{casimir5}.
Thus by \eqref{est-casimir1}  it follows that
\[
\tau(\Omega)-\pi_{\xi,\lambda}(\Omega)\ge \delta.
\]
Using \eqref{casimir5} we obtain
\[
\tau(\Omega)-c(\xi)\ge \delta-\|\lambda\|^2
\]
for every $\lambda\in\af^\star$. Hence $\tau(\Omega)-c(\xi)\ge \delta$, which
proves the lemma.
\end{proof}

Given $\lambda>0$, let
\[
\Pi_{\di}(M(\A);\lambda):=\left\{\pi\in\Pi_{\di}(M(\A))\colon 
|\lambda_{\pi_\infty}|\le\lambda\right\}.
\]
Let $d=\dim M(\R)^1/K_{M,\infty}$. As in \cite[Proposition 3.5]{Mu1} it follows 
that for every 
$\nu\in\Pi(K_\infty)$ there exists $C>0$ such that
\begin{equation}\label{estim10}
\sum_{\pi\in\Pi_{\di}(M(\A))_\lambda}\dim\cA^2_\pi(P)^{K_f,\nu}\le C(1+\lambda^{d/2})
\end{equation}
for all $\lambda\ge 0$.

Put
\[
\cA^2_\pi(P)^{K_f,\cT}=\bigoplus_{\nu\in\cT}\cA^2_\pi(P)^{K_f,\nu},
\]
where $\cT$ is defined by \eqref{def-F}.

Let $\delta>0$ be as in Lemma \ref{est-casim}. Put $c=\delta/2$. It follows from
Lemma \ref{est-casim} that for $t\ge 1$, \eqref{est-specside1} can be estimated
by
\begin{equation}\label{est-specside2}
e^{-ct}\sum_{\pi\in\Pi_{\di}(M(\A))}
\sum_{\nu\in\cT}\dim\left(\cA^2_\pi(P)^{K_f,\nu}\right)e^{-t(\tau(\Omega)-c(\pi_\infty))/2}
(1+|\lambda_{\pi_\infty}|)^m\level(K_f;G_M^+)^\varepsilon,
\end{equation}
where $m\in\N$ is sufficiently large. 
Now observe that $\tau(\Omega)\ge 0$. Thus by \eqref{casimir4} we get
\begin{equation}\label{est-tau}
\tau(\Omega)-c(\pi_\infty)\ge -\lambda_{\pi_\infty}.
\end{equation}
By \cite[Lemma 13.2]{MzM} there are only finitely many 
$\pi\in \Pi_{\di}(M(\A))$ with $\cA^2_\pi(P)^{K_f,\cT}\neq 0$ and 
$-\lambda_{\pi_\infty}\le 0$. Decompose the sum over $\pi$ in 
\eqref{est-specside2} in two summands $\Sigma_1(t)$ and $\Sigma_2(t)$,
where in $\Sigma_1(t)$ the summation runs over all $\pi$ with 
$-\lambda_{\pi_\infty}\le 0$. Using \eqref{est-tau} it follows that for 
$-\lambda_{\pi_\infty}> 0$ we have
\[
\tau(\Omega)-c(\pi_\infty)\ge |\lambda_{\pi_\infty}|
\]
Thus for every $l\in\N$, $K_f$ and $t\ge 1$ we have
\begin{equation}\label{sum2}
\Sigma_2(t)\ll_l e^{-ct}\sum_{\substack{\pi\in\Pi_{\di}(M(\A))\\-\lambda_{\pi_\infty}
>0}}\sum_{\nu\in\cT}\dim\left(\cA^2_\pi(P)^{K_f,\nu}\right)
(1+|\lambda_{\pi_\infty}|)^{-l}\level(K_f,G_M^+)^\varepsilon.
\end{equation}
To estimate $\Sigma_1(t)$ we need the following lemma. 
\begin{lem}\label{lem-finite}
Let $\nu\in\Pi(K_\infty)$. There exists $C_1\in\R$ such that $C_1\le 
-\lambda_{\pi_\infty}$ for all $\pi\in\Pi_{\di}(M(\A))$ with 
$\cA^2_\pi(P)^{K_f,\nu}\neq 0$ for some open compact subgroup $K_f$ of $G(\A_f)$.
\end{lem}
\begin{proof}
Let $\pi\in\Pi_{\di}(M(\A))$. Let $K_f$ be an open compact subgroup of $G(\A_f)$
such that $\cA^2_\pi(P)^{K_f,\nu}\neq 0$. Let $\cH_P(\pi_\infty)$ (resp.
$\cH_P(\pi_f)$) be the Hilbert space of the
induced representation $I^{G(\R)}_{P(\R)}(\pi_\infty)$ 
(resp. $I_{P(\A_f)}^{G(\A_f)}(\pi_f)$). Let $\nu\in\Pi(K_\infty)$. Denote by  
$\cH_P(\pi_\infty)_\nu$ the $\nu$-isotypical subspace of $\cH_P(\pi_\infty)$.
Then by \cite[(3.5)]{Mu1} it follows that 
\begin{equation}\label{nicht0}
\dim (\cH_P(\pi_f)^{K_f})\neq0\;\text{and}\;\dim(\cH_P(\pi_\infty)_\nu)\neq0.
\end{equation}
Using Frobenius reciprocity \cite[p. 208]{Kn} we get
\[
[I^{G(\R)}_{P(\R)}(\pi_\infty)|_{K_\infty}\colon\nu]=\sum_{\tau\in\Pi(K_{M,\infty})}
[\pi_\infty|_{K_{M,\infty}}\colon\tau]\cdot[\nu|_{K_{M,\infty}}\colon\tau].
\]
This implies 
\[
\dim(\cH_P(\pi_\infty)_\nu)\le \dim(\nu)\sum_{\tau\in\Pi(K_{M,\infty})}
\dim(\cH_{\pi_\infty}(\tau))[\nu|_{K_{M,\infty}}\colon\tau].
\]
By \eqref{nicht0} it follows that there exists $\tau\in\Pi(K_{M,\infty})$ such
that 
\begin{equation}\label{nicht0a}
[\nu|_{K_{M,\infty}}\colon\tau]\neq0\;\;\text{and}\;\;\dim(\cH_{\pi_\infty}(\tau))\neq0.
\end{equation}
Replacing $K_f$ by a subgroup of finite index if necessary, we can assume 
that $K_f$ is of the form $K_f=\prod_{p<\infty} K_p$.  For any $p<\infty$ denote by $\cH_P(\pi_p)$ the Hilbert space of the
induced representation $I_{P(\Q_p)}^{G(\Q_p)}(\pi_p)$. Let $\cH_P(\pi_p)^{K_p}$ be the
subspace of $K_p$-invariant vectors. Then $\dim\cH_P(\pi_p)^{K_p}=1$ for almost
all $p$ and
\[
\cH_P(\pi_f)^{K_f}\cong\bigotimes_{p<\infty}\cH_P(\pi_p)^{K_p}.
\]
Let $K_{M,f}:=K_f\cap M(\A_f)$. Then $K_{M,f}$ is an open compact subgroup of 
$M(\A_f)$. Using \cite[(3.7)]{Mu1} and \eqref{nicht0},
it follows that $\dim\cH_{\pi_f}^{K_{M,f}}\neq0$. Now recall that
there exist arithmetic subgroups $\Gamma_{M,i}\subset M(\R)$, $i=1,\dots,l$, 
such that 
\[
M(\Q)\bs M(\A)/K_{M.f}\cong \bigsqcup_{i=1}^l(\Gamma_{M,i}\bs M(\R))
\]
(cf. \cite[\S3]{MzM}). Hence
\[
L^2(A_M(\R)^0M(\Q)\bs M(\A))^{K_{M,f}}\cong \bigoplus_{i=1}^l 
L^2(A_M(\R)^0\Gamma_{M,i}\bs M(\R))
\]
as $M(\R)$-modules. The condition $\dim\cH_{\pi_f}^{K_{M,f}}\neq0$ implies 
that $\pi_\infty$ occurs as an
irreducible subrepresentation of the right regular representation of $M(\R)$
in $L^2(A_M(\R)^0\Gamma_{M,i}\bs M(\R))$ for some $i=1,\dots,l$. Put $\Gamma_M:=
\Gamma_{M,i}$. Let $\Omega_{M(\R)^1}$ be the Casimir element of $M(\R)^1$. Given
$\tau\in\Pi(K_{M,\infty})$, let $A_\tau$ be the differential operator in
$C^\infty(\Gamma_M\bs M(\R)^1;\tau)$ which is induced by $-\Omega_{M(\R)^1}$.
Let $\bar A_\tau$ be the self-adjoint extension of $A_\tau$ in $L^2$. Assume
that $\tau$ satisfies \eqref{nicht0a}. Then it follows that $-\lambda_{\pi_\infty}$
is an eigenvalue of $\bar A_\tau$, acting in $L^2(\Gamma_M\bs M(\R)^1;\tau)$. 
Now let $\Delta_\tau$ be the Bochner-Laplace operator, acting in the same
Hilbert space. Let $\Lambda_\tau$ be the Casimir eigenvalue of $\tau$. We have
\[
\bar A_\tau=\Delta_\tau-\Lambda_\tau\Id
\]
(cf. \cite[Proposition 1.1]{Mia}). Furthermore note that $\Lambda_\tau\ge 0$
and $\Delta_\tau\ge 0$. Thus it follows that 
$-\lambda_{\pi_\infty}\ge -\Lambda_\tau$. Let
\[
C_1=-\max\{\Lambda_\tau\colon \tau\in\Pi(K_{M,\infty}),\;\;[\nu|_{K_{M,\infty}}\colon
\tau]>0\}.
\]
Then the lemma holds with this choice of $C_1$.
\end{proof}
It follows from Lemma \ref{lem-finite} that there exists $C_1\in\R$, which 
depends on $\cT$, but is independent of $K_f$, such that 
$C_1\le-\lambda_{\pi_\infty}$ for all
$\pi\in \Pi_{\di}(M(\A))$ with $\cA^2_\pi(P)^{K_f,\cT}\neq 0$. Thus for every
$l\in\N$, $K_f$, and $t\ge 1$ we get
\begin{equation}\label{sum1}
\Sigma_1(t)\ll_l  e^{-ct}\sum_{\substack{\pi\in\Pi_{\di}(M(\A))\\-\lambda_{\pi_\infty}
\le 0}}\sum_{\nu\in\cT}\dim\left(\cA^2_\pi(P)^{K_f,\nu}\right)
(1+|\lambda_{\pi_\infty}|)^{-l}\level(K_f,G_M^+)^\varepsilon.
\end{equation}

Putting everything together we obtain the following lemma.
\begin{lem}\label{lem-trspec}
Suppose that $G$ satisfies properties (TWN) \cite[Definition 5.2]{FLM2} and (BD)
\cite[Definition 5.9]{FLM2}. Let $\tau\in \Rep(G(\R))$. Assume that $\tau\ncong\tau_\theta$. Let $M$ be a
proper Levi subgroup of $G$. There exists $c>0$, independent of $K_f$,
 and for every $l\in\N$ and
$\varepsilon>0$ there exists $C>0$, which is independent of $K_f$, such that
\[
|J_{\spec,M}(\phi_t^{\tau,p})|\le C e^{-ct}
\sum_{\pi\in\Pi_{\di}(M(\A))}\sum_{\nu\in\cT}\dim\left(\cA^2_\pi(P)^{K_f,\nu}\right)
(1+|\lambda_{\pi_\infty}|)^{-l}\level(K_f,G_M^+)^\varepsilon.
\]
for $t\ge 1$ and $p=0,\dots,d$. 
\end{lem}

We now specialize to the case of principal congruence subgroups. Fix a faithful
$\Q$-rational representation $\rho\colon G\to\GL(V)$ and a lattice $\Lambda$
in the representation space $V$ such that the stabilizer of $\widehat\Lambda =
\widehat \Z\otimes\Lambda \subset \A_f\otimes V$ in $G(\A_f)$ is the group $K_f$.
Since the maximal compact subgroups of $\GL(\A_f\otimes V)$ are precisely the
stabilizers of lattices, it is easy to see that such a lattice exists. For
$N\in\N$ let 
\[
K(N)=\{g\in G(\A_f)\colon \rho(g)v\equiv v \mod N\widehat\Lambda,\;v\in V\}
\]
be the principal congruence subgroup of level $N$, which is a factorizable
normal open subgroup of $\K_f$. Let 
\begin{equation}\label{adelic-quot}
Y(N):=G(\Q)\bs(\widetilde X\times G(\A_f)/K(N)
\end{equation}
be the adelic quotient associated to $K(N)$. Fix $P=M\cdot U\in \cP(M)$. By 
\eqref{ind-repr} have 
\begin{equation}\label{dim-auto1}
\begin{split}
\dim\cA^2_\pi(P)^{K(N),\nu}&=m_\pi\dim\Ind_{P(\A)}^{G(\A)}(\pi)^{(K(N),\nu)}\\
&=m_\pi \dim\Ind_{P(\R)}^{G(\R)}(\pi_\infty)^\nu \dim\Ind_{P(\A_f)}^{G(\A_f)}
(\pi_f)^{K(N)}.
\end{split}
\end{equation}
Note that $\dim\Ind_{P(\R)}^{G(\R)}(\pi_\infty)^\nu$ is bounded by $(\dim\nu)^2$.
Let $\K_f\subset G(\A_f)$ be the standard maximal compact subgroup. Let
$\Xi$ be a set of coset representatives for the double cosets
$(P(\A_f)\cap \K_f)\bs \K_f/K(N)$. Since
$K(N)$ is  a normal subgroup of $\K_f$ of finite index, it follows from
\cite[Lemme, III.2]{Re} that the map $\varphi\mapsto 
\left(\varphi(g)\right)_{g\in\Xi}$
defines an isomorphism
\[
\Ind_{P(\A_f)}^{G(\A_f)}(\pi_f)^{K(N)}\cong\oplus_{g\in\Xi}(\pi_f)^{P(\A_f)\cap K(N)}.
\]
Thus we get
\[
\dim\Ind_{P(\A_f)}^{G(\A_f)}(\pi_f)^{K(N)}\le [\K_f\colon (\K_f\cap P(\A_f))K(N)]
\dim(\pi_f^{K_M(N)}).
\] 
Using the factorization $\K_f\cap P(\A_f)=(K_f\cap M(\A_f))(\K_f\cap U(\A_f))$,
we can write
\[
\begin{split}
[\K_f\colon (\K_f\cap& P(\A_f))K(N)]=\vol(K_M(N))\vol(K(N))^{-1}
[\K_f\cap U(\A_f)\colon K(N)\cap U(\A_f)]^{-1}\\
&[K(N)\cap P(\A_f)\colon (K(N)\cap M(\A_f))(K(N)\cap U(\A_f)].
\end{split}
\]
The index $[K(N)\cap P(\A_f)\colon(K(N)\cap M(\A_f))(K(N)\cap U(\A_f)]$ is 
bounded independently of $N$. Furthermore, identifying $U$ with its Lie algebra
$\uf$ via the exponential map, which is an isomorphism of affine varieties,
it follows that there exist $C_1,C_2>0$ such that 
\[
C_1N^{-\dim U}\le [\K_f\cap U(\A_f)\colon K(N)\cap U(\A_f)]^{-1}\le C_2 N^{-\dim U}
\]
for all $N\in\N$. Therefore there exist $C>0$, independent of $N$, such that
\[
\Ind_{P(\A_f)}^{G(\A_f)}(\pi_f)^{K(N)}\le C N^{-\dim U}\vol(K(N))^{-1}\vol(K_M(N))
\dim\pi_f^{K_M(N)}.
\]
Let
\begin{equation}\label{test-fct5}
\phi_{t,N}^{\tau,p}=h_t^{\tau,p}\otimes \chi_{K(N)}.
\end{equation}
Then $\phi_{t,N}^{\tau,p}\in \Co(G(\A)^1,K(N))$. 
Combined with Lemma \ref{lem-trspec} and \eqref{dim-auto1} it follows that
there exists $C>0$ such that
\begin{equation}\label{est-trspec3}
\begin{split}
\frac{1}{\vol(Y(N))}|J_{\spec,M}(\phi_{t,N}^{\tau,p})|\le &C e^{-ct}
N^{-\dim U+\varepsilon}\\
&\cdot\vol(K_M(N))\sum_{\pi\in\Pi_{\di}(M(\A))^\cT}m_\pi 
(1+|\lambda_{\pi_\infty}|)^{-l}\dim \pi_f^{K_M(N)}.
\end{split}
\end{equation}
for all $t\ge 1$ and $N\in\N$. Here $\Pi_{\di}(M(\A))^\cT$ denotes the set of all $\pi\in\Pi_{\di}(M(\A))$ such that there exists $\nu\in\cT$ with $\cA_\pi^2(P)^\nu\neq0$. For an open compact subgroup 
$K_{M,f}\subset M(\A_f)$ let
$\mu^M_{K_f}$ be the measure on $\Pi(M(\R)^1)$ defined by
\[
\begin{split}
\mu^M_{K_f}=&\frac{\vol(K_{M,f})}{\vol(M(\Q)\bs M(\A)^1)}\\
&\hskip20pt\cdot\sum_{\pi\in\Pi(M(\A)^1)}\dim\Hom_{M(\A)^1}(\pi,L^2(M(\Q)\bs M(\A)^1))
\dim \pi_f^{K_{M,f}}\delta_{\pi_\infty}.
\end{split}
\]
It follows from \cite[Lemma 7.7]{FLM2},
together with \cite[Proposition 5.5]{FLM2} and \cite[Theorem 5.15]{FLM2} that
the collection of measures $\{\mu^M_{K_M(N)}\}_{N\in\N}$ is polynomially bounded
in the sense of \cite[Definition 6.2]{FLM2}. For $l\in\N$ let $g_{l,\cT}$ be
the non-negative function on $\Pi(G(\R))$ defined by
\[
g_{l,\cT}(\pi):=\begin{cases} (1+|\lambda_\pi|)^{-l},& \text{if}\;\pi\in
\Pi(G(\R))^\cT,\\ 0,& \text{otherwise}.
\end{cases}
\]
Then it follows from 
 \cite[Proposition 6.1, (4)]{FLM2} that there exists $l\in\N$, which depends
only on $\cT$,  such that
\begin{equation}\label{polynbd}
\mu^M_{K_f}(g_{l,\cT})=\frac{\vol(K_M(N))}{\vol(M(\Q)\bs M(\A)^1)}
\sum_{\pi\in\Pi_{\di}(M(\A))^{\cT}}(1+|\lambda_{\pi_\infty}|)^{-l} m_\pi 
\dim \pi_f^{K_M(N)}
\end{equation}
is bounded independently of $N\in\N$. Together with \eqref{est-trspec3} we 
obtain the following lemma.
\begin{lem}\label{lem-trspec1}
Suppose that $G$ satisfies properties (TWN) \cite[Definition 5.2]{FLM2} and
(BD) \cite[Definition 5.9]{FLM2}. 
Let $M\in\cL$, $M\neq G$. Let $P=M\cdot U\in\cP(M)$ and let $\tau\in
\Rep(G(\R))$ such that $\tau\ncong\tau_\theta$. 
There exist $C,c,\delta>0$ such that
\begin{equation}\label{est-trspec4}
\frac{1}{\vol(Y(N))}|J_{\spec,M}(\phi_{t,N}^{\tau,p})|\le C e^{-ct} 
N^{-\delta}
\end{equation}
for all $t\ge 1$, $p=0,\dots,d$, and $N\in\N$.
\end{lem}
Now we consider the case $M=G$. Then by definition of $\phi_{t,N}^{\tau,p}$ we have
\begin{equation}\label{spec-G1}
\begin{split}
J_{\spec,G}(\phi_{t,N}^{\tau,p})=
\sum_{\pi\in\Pi_{\di}(G(\A)^1)} m_\pi\Tr\pi(\phi_{t,N}^{\tau,p})
=\sum_{\pi\in\Pi_{\di}(G(\A)^1)} m_\pi\dim(\pi_f^{K(N)})\Tr\pi_\infty(h_t^{\tau,p}).
\end{split}
\end{equation}
Now observe that by \cite[(4.18), (4.19)]{MP2} we have
\[
\Tr\pi_\infty(h_t^{\tau,p})=e^{t(\pi_\infty(\Omega)-\tau(\Omega))}\dim(\cH_{\pi_\infty}\otimes
\Lambda^p\pf^\star\otimes V_\tau)^{K_\infty}.
\]
Furthermore, for $\nu\in\Pi(K_\infty)$ we have
\[
[\pi_\infty|_{K_\infty}\colon\nu]\le \dim \nu
\]
(see \cite[Theorem 8.1]{Kn}). Thus there exists $C>0$ such that
\[
\frac{1}{\vol(Y(N))}|J_{\spec,G}(\phi_{t,N}^{\tau,p})|\le C \vol(K(N))
\sum_{\pi\in\Pi_{\di}(G(\A)^1)^{\cT}} m_\pi\dim(\pi_f^{K(N)}) 
e^{t(\pi_\infty(\Omega)-\tau(\Omega))}
\]
for all $t>0$ and $N\in\N$. As above, put $\lambda_{\pi_\infty}= 
\pi_\infty(\Omega)$. If we argue as in the proof of Lemma \ref{lem-trspec},
it follows that there exists $c>0$ and for all $l\in\N$ there exist $C_l>0$ 
such that
\begin{equation}
\begin{split}
\frac{1}{\vol(Y(N))}|J_{\spec,G}(\phi_{t,N}^{\tau,p})|\le &C_l e^{-ct}\vol(K(N))\\
&\cdot\sum_{\pi\in\Pi_{\di}(G(\A)^1)^{\cT}} m_\pi(1+|\lambda_{\pi_\infty}|)^{-l}
\dim(\pi_f^{K(N)}) 
\end{split}
\end{equation} 
for all $t\ge 1$ and $N\in\N$. Using that \eqref{polynbd} for $M=G$ , we get the following:
\begin{lem}\label{lem-trspec2}
Let $\tau\in\Rep(G(\R))$ such that $\tau\ncong\tau_\theta$. There exist
$C,c>0$ such that
\begin{equation}\label{est-trspec5}
\frac{1}{\vol(Y(N))}|J_{\spec,G}(\phi_{t,N}^{\tau,p})|\le C e^{-ct}
\end{equation}
for all $t\ge 1$, $p=0,\dots,d$, and $N\in \N$.
\end{lem}
Combining Lemmas \ref{lem-trspec1}, Lemma \ref{lem-trspec2} and 
\eqref{specside1} it 
follows that there exist $C,c>0$ such that
\begin{equation}\label{est-specside5}
\frac{1}{\vol(Y(N))}|J_{\spec}(\phi_{t,N}^{\tau,p})|\le C e^{-ct}
\end{equation}
for all $t\ge 1$, $p=0,\dots,d$, and $N\in \N$.

Let $\Delta_{p,N}(\tau)$ be the Laplace operator on $E_\tau$-valued $p$-forms.
By \eqref{regtrace1a} we have
\[
\Tr_{\reg}\left(e^{-t\Delta_{p,N}(\tau)}\right)=J_{\spec}(\phi_{t,N}^{\tau,p})
\]
and by \eqref{est-specside5} we obtain the following bound:
\begin{prop}\label{prop-larget}
Suppose that $G$ satisfies properties (TWN) \cite[Definition 5.2]{FLM2} and
(BD) \cite[Definition 5.9]{FLM2}. There exist $C,c>0$ such that
\[
\frac{1}{\vol(Y(N))}|\Tr_{\reg}\left(e^{-t\Delta_{p,N}(\tau)}\right)|\le C e^{-ct}
\]
for all $t\ge 1$, $p=0,\dots,d$, and $N\in \N$.
\end{prop}
Recall that by \cite[Prop. 5.5, Prop 5.15]{FLM2} the properties (TWN) and
(BD) are satisfied for $\GL(n)$ and $\SL(n)$. Hence we get the following 
corollary.
\begin{cor}\label{cor-larget}
Let $G=\GL(n)$ or $\SL(n)$. There exist $C,c>0$ such that
\[
\frac{1}{\vol(Y(N))}|\Tr_{\reg}\left(e^{-t\Delta_{p,N}(\tau)}\right)|\le C e^{-ct}
\]
for all $t\ge 1$, $p=0,\dots,d$, and $N\in \N$.
\end{cor}

\section{Modification of the heat kernel}
\label{sec-short-time}
\setcounter{equation}{0}

In order to study the short time behavior of the regularized trace of 
the heat operator with the help of the trace formula, we need to show that 
we can replace $h_t^{\tau,p}$ by an appropriate compactly supported test 
function without changing the asymptotic behavior as $t\to0$. We introduced 
such a modification of $h_t^{\tau,p}$ already in \cite{MzM}. The main purpose of 
this section is to establish estimations which are uniform in the lattice. 

In this section we assume that $G=\GL(n)$. Let $G(\R)^1$ be defined by 
\eqref{gr1}.
Let $d(x,y)$ denote the geodesic distance 
of $x,y\in \widetilde X$. On $G(\R)^1$ we introduce the function $r$ by  
\[
r(g):=d(gK_\infty,K_\infty),\quad g\in G(\R)^1.
\]
For $R>0$ let
\begin{equation}\label{ball}
B_R:=\{g\in G(\R)^1\colon r(g)< R\}.
\end{equation}
We need the following auxiliary lemma.
\begin{lem}\label{volume}
There exist $C,c>0$ such that
\[
\int_{G(\R)^1} e^{-r^2(g)/t} dg\le C e^{ct}
\]
for $t>0$.
\end{lem}
\begin{proof}
Note that $r(g)$ is bi-$K_\infty$-invariant. Thus using the Cartan 
decomposition $G(\R)^1=K_\infty A^+ K_\infty$, we get
\[
\int_{G(\R)^1} e^{-r^2(g)/t} dg=\int_{A^+} e^{-r^2(a)/t}\delta(a)da,
\]
where 
\[
\delta(\exp H)=\prod_{\alpha\in\Delta+}(\sinh \alpha(H))^{m_\alpha},\quad H\in\af^+
\]
(see \cite[Chapt. I, Theorem 5.8]{He}). Let $a=\diag(e^{H_1}, \ldots, e^{H_n})\in A^+$ so that $H_i-H_{i+1}>0$ for $i=1,\ldots, n-1$ and $\sum_{i=1}^n H_i=0$. 
Moreover, 
\[
 r(a)^2= H_1^2+\ldots+H_n^2
\]
by \cite[Corollary 10.42]{BH}. Note that there exists a constant $c>0$ such that $\delta(\exp H)\ll e^{c\|H\|}$ for every $H\in\af^+$. Hence it suffices to find an upper bound for $\int_0^\infty e^{-r^2/t} e^{c r}\, dr$.  
Note that
\[
\int_0^\infty e^{-r^2/t}e^{cr} dr=\frac{\sqrt{\pi t}}{2}
\exp(c^2t)(1-\erf(c\sqrt{t})),
\]
where $\erf(x)$ is the error function (see \cite[3.322,2]{GR}). This proves the 
claim.
\end{proof}

Let 
$f\in C^\infty(\R)$ such that $f(u)=1$, if $|u|\le 1/2$, and $f(u)=0$, if
$|u|\ge 1$. Let $\varphi_R\in C^\infty_c(G(\R)^1)$ be defined by
\begin{equation}\label{phiR}
\varphi_R(g):= f\left(\frac{r(g)}{R}\right).
\end{equation}
Then we have $\supp\varphi_R\subset B_R$. Extend $\varphi_R$ to $G(\R)$ by
\[
\varphi_R(g_\infty z)=\varphi_R(g_\infty),\quad g_\infty\in G(\R)^1,\; z\in 
A_G(\R)^0.
\]
Define $\widetilde h^{\tau,p}_{t,R}\in C^\infty(G(\R))$ by
\begin{equation}\label{test-funct1}
\widetilde h^{\tau,p}_{t,R}(g_\infty):= \varphi_R(g_\infty) h^{\tau,p}_t(g_\infty),
\quad g_\infty\in G(\R).
\end{equation}
Then the restriction of 
$\widetilde h_{t,R}^{\tau,p}\otimes\chi_{K(N)}$ to $G(\A)^1$ belongs to 
$C^\infty_c(G(\A)^1)$. Let $K(N)\subset\GL(n,\A_f)$ be the principal
congruence subgroup of level $N$ and let $Y(N)$ be the adelic quotient 
defined by \eqref{adelic-quot}.

\begin{prop}\label{prop-cptsupp}
There exist constants $C_1,C_2,C_3>0$ such that
\[
\frac{1}{\vol(Y(N))}\big|J_{\spec}(h^{\tau,p}_t\otimes\chi_{K(N)})- 
J_{\spec}(\widetilde h^{\tau,p}_{t,R}\otimes\chi_{K(N)})\big|\le C_1 e^{-C_2R^2/t+C_3t}
\]
for all $N\in\N$, $p=0,\dots,d$, $t>0$ and $R\ge 1$.
\end{prop}
Proposition \ref{prop-cptsupp} allows us to replace $h^{\tau,p}_t$ by a compactly
supported function. 
\begin{proof}
Let
$\psi_R:=1-\varphi_R$.  Then 
\[
J_{\spec}(h^{\tau,p}_t\otimes \chi_{K(N)})- 
J_{\spec}(\widetilde h^{\tau,p}_{t,R}\otimes\chi_{K(N)})=
J_{\spec}(\psi_R h^{\tau,p}_t\otimes\chi_{K(N)}).
\]
Now we use the refined spectral expansion \eqref{specside2}. Let $M\in\cL$ and
let $J_{\spec,M}$ be the distribution on the right hand side of \eqref{specside2},
which corresponds to $M$. Let 
\[
\Delta_G=-\Omega+2\Omega_{K_\infty},
\]
where $\Omega$ (resp. $\Omega_{K_\infty}$) denotes the Casimir operator of 
$G(\R)^1$(resp. $K_\infty$). Observe that $\psi_R h_t^{\tau,p}\otimes\chi_{K(N)}$ 
belongs to 
$\Co(G(\A)^1)$ and the proof of Lemma 7.2 and Corollary 7.4 in \cite{FLM1} 
extends to $h\in\Co(G(\A)^1)$. Thus there exists $k\ge 1$ such that for any 
$\varepsilon>0$ we have
\begin{equation}\label{estim-JM}
\begin{split}
\frac{1}{\vol(Y(N))}J_{\spec,M}(\psi_R h_t^{\tau,p}\otimes\chi_{K(N)})&=
\frac{1}{\vol(G(\Q)\bs G(\A)^1)}
J_{\spec,M}(\psi_R h^{\tau,p}_t\otimes \1_{K(N)})\\
&\ll_{\cT,\varepsilon}
\|(\Id+\Delta_G)^k(\psi_R h^{\tau,p}_t)\|_{L^1(G(\R)^1)} 
N^{(\dim M-\dim G)/2+\varepsilon}
\end{split}
\end{equation}
for all $N\in\N$, $p=0,\dots,d$, $t>0$, and $R>0$.

Let $\gf$ be the Lie algebra of $G(\R)^1$ and let $Y_1,\dots,Y_r$ be an 
orthonormal basis of $\gf$. Then $\Delta_G=-\sum_i Y_i^2$. Denote by $\nabla$
the canonical connection on $G(\R)^1$. Then it follows that there exists 
$C>0$ such that
\[
|(\Id+\Delta_G)^kh(g)|\le C\sum_{l=0}^{2k}\|\nabla^lh(g)\|, \quad g\in G(\R)^1,
\]
for all $h\in C^\infty(G(\R)^1)$. Let $m=\dim G(\R)^1$. 
By \cite[Proposition 2.1]{Mu1} it follows that for every $T>0$ and
$j\in\N$  there exist $C_2,C_3>0$ such that
\begin{equation*}
\|\nabla^j h_t^{\tau,p}(g)\|\le C_2 t^{-(m+j)/2} e^{-C_3r^2(g)/t},\quad g\in G(\R)^1,
\end{equation*}
for all $0<t\le T$. Using the semigroup property and arguing as in the proof of 
Corollary 1.6 in \cite{Do2}, it follows that there exist $A_1,A_2,A_3>0$
such that
\begin{equation}\label{connect}
\|\nabla^j h_t^{\tau,p}(g)\|\le A_1 t^{-(m+j)/2} e^{-A_2r^2(g)/t+A_3t},\quad g\in G(\R)^1,
\end{equation}
for all $t>0$. Now observe that for every $j\in\N$ there exists $C_j>0$ such
that
\[
\|\nabla^j\psi_R\|\le C_j
\]
for all $R\ge 1$. Since $\psi_R$ vanishes on $B_R$, it follows from 
\eqref{connect}  that there exist $C_4,C_5,C_6>0$ such 
that 
\[
\sum_{l=0}^{2k}\|\nabla^
l(\psi_R h_t^{\tau,p})(g)\|\le 
C_4 e^{-C_5R^2/t+A_3t}e^{-C_6r^2(g)/t}
\]
for all $g\in G(\R)^1$, $t>0$, and $R\ge 1$. Using Lemma \ref{volume},
it follows that there exist $C_1,C_2,C_3>0$ such that
\begin{equation}\label{estim6}
\|(\Id+\Delta_G)^k(\psi_R h_t^{\tau,p})\|_{L^1(G(\R)^1)}\le C_1 e^{-C_2R^2/t+C_3t}
\end{equation}
for all $t>0$ and $R\ge 1$.  Combined with 
\eqref{estim-JM} it follows that for every $\varepsilon>0$ we have
\[
\frac{1}{\vol(Y(N))}J_{\spec,M}(\psi_R h^{\tau,p}_t\otimes \chi_{K(N)})\ll_\varepsilon 
e^{-cR^2/t}N^{(\dim M -\dim G)/2+\varepsilon}
\]
for all $N\in\N$, $p=0,\dots,d$, and $t>0$. and $R\ge 1$. Especially, there 
exist $C_1,C_2,C_3>0$ such that
\begin{equation}\label{estim-JM1}
\frac{1}{\vol(Y(N))}|J_{\spec,M}(\psi_R h^{\tau,p}_t\otimes\chi_{K(N)})|\le 
C_1 e^{-C_2R^2/t+C_3t}
\end{equation}
for all $N\in\N$, $p=0,\dots,d$, and $t>0$, and $R\ge 1$.

It remains to consider the case $M=G$. Then we have 
\[
\begin{split}
J_{\spec,G}(\psi_R h_t^{\tau,p}\otimes\chi_{K(N)})&=\sum_{\pi\in\Pi_{\di}(G(\A)^1)} m_\pi
\Tr\pi(\psi_R h_t^{\tau,p}\otimes\chi_{K(N)})\\
&=\sum_{\pi\in\Pi_{\di}(G(\A)^1)} m_\pi\dim(\pi_f^{K(N)})
\Tr\pi_\infty(\psi_R h_t^{\tau,p}).
\end{split}
\]
For $\nu\in\Pi(K_\infty)$ denote by $\H_{\pi_\infty}(\nu)$ the $\nu$-isotypic
subspace. Let
\[
\H_{\pi_\infty}^\cT=\sum_{\nu\in\cT}\H_{\pi_\infty}(\nu).
\]
Then for every $k\in\N$ we have
\[
|\Tr\pi_\infty(\psi_R h^{\tau,p}_t)|\le 
\|(\Id+\pi_\infty(\Delta_G))^{-k}\|_{1,\H_{\pi_\infty}^\cT} 
\|(\Id+\Delta_G)^k(\psi_R h^{\tau,p}_t)\|_{L^1(G(\A)^1)}.
\]
Now observe that $\pi_\infty(\Delta_G)$ acts on $\H_{\pi_\infty}(\nu)$ by the 
scalar $-\lambda_{\pi_\infty}+2\lambda_\nu$, where $\lambda_{\pi_\infty}$ and 
$\lambda_\nu$ are the Casimir eigenvalues of $\pi_\infty$ and $\pi_\nu$, 
respectively. Furthermore, by \cite[Lemma 6.1]{Mu2} we have
\begin{equation}\label{casimir1}
-\lambda_{\pi_\infty}+\lambda_\nu\ge 0
\end{equation}
for $\H_{\pi_f}^{K(N)}\neq 0$ and $\H_{\pi_\infty}(\nu)\neq 0$. 
Moreover $\lambda_\nu\ge 0$. Thus $1-\lambda_{\pi_\infty}+2\lambda_\nu>0$ and we get
\[
\|(\Id+\pi_\infty(\Delta_G))^{-k}\|_{1,\H_{\pi_\infty}^\cT}\le 
\sum_{\nu\in\cT}\dim(\nu)(1-\lambda_{\pi_\infty}+2\lambda_\nu)^{-k}.
\]
Using \eqref{casimir1} we get 
\[
(1-\lambda_{\pi_\infty}+2\lambda_\nu)^2\ge 
\frac{1}{4}(1+\lambda_{\pi_\infty}^2+\lambda_\nu^2)\ge\frac{1}{4} 
(1+|\lambda_{\pi_\infty}|)^2.
\]
Thus we get
\[
\|(\Id+\pi_\infty(\Delta_G))^{-k}\|_{1,\H_{\pi_\infty}^\cT}\le 
\frac{1}{4}\dim(\H_{\pi_\infty}^\cT)(1+|\lambda_{\pi_\infty}|)^{-k}.
\]
Together with \eqref{estim6} it follows that for every $k\in\N$ there exists
$C_k>0$ such that
\[
|\Tr\pi_\infty(\psi_R h^{\tau,p}_t)|\le 
C_k e^{-C_2R^2/t+C_3t}(1+|\lambda_{\pi_\infty}|)^{-k}
\]
for all $t>0$ and $R\ge 1$.
This gives
\[
\begin{split}
\frac{1}{\vol(Y(N))}&|J_{\spec,G}(\psi_R h_t^{\tau.p}\otimes\chi_{K(N)})|\\
&\le
C_k e^{-C_2R^2/t+C_3t}\vol(K(N))\sum_{\pi\in\Pi{\di}(G(\A)^1)} m_\pi\dim(\pi_f^{K(N)})
(1+|\lambda_{\pi_\infty}|)^{-k}
\end{split}
\]
for all $t>0$ and $R\ge 1$.
As above it follows from \cite[Proposition 6.1, (4)]{FLM2} that there exists
$k\in\N$, which depends only on $\cT$, such that $\vol(K(N))$
times the sum is bounded independently of $N\in\N$. Hence there exist 
$C_1,C_2,C_3>0$ such that
\[
\frac{1}{\vol(Y(N))}|J_{\spec,G}(\psi_R h^{\tau,p}_{t}\otimes\chi_{K(N)})|\le 
C_1 e^{-C_2R^2/t+C_3t}
\]
for all $t>0$, $p=0,\dots,d$, $N\in\N$, and $R\ge 1$. This completes the proof 
of the proposition.
\end{proof}

\section{The geometric side of the trace formula}
\label{sec-geomside}
\setcounter{equation}{0}

In this section we assume that $G=\GL(n)$. 
To study the behavior of the regularized trace for small time, we use the
geometric side $J_{\geo}$ of the Arthur trace formula. Consider the equivalence 
relation on
$G(\Q)$ defined by $\gamma\sim\gamma^\prime$ whenever the semisimple parts of
$\gamma$ and $\gamma^\prime$ are $G(\Q)$-conjugate, and denote by $\cO_G$ the set
of all resulting equivalence classes. They are indexed by the conjugacy classes
of semisimple elements of $G(\Q)$. Then the coarse geometric expansion of
$J_{\geo}$ is 
\begin{equation}\label{geo-side1}
J_{\geo}(f)=\sum_{\of\in\cO_G} J_{\of}(f),\quad f\in C^\infty_c(G(\A)^1),
\end{equation}
where the distributions are the value at $T=0$ of a polynomial $J^T_{\of}(f)$
defined in \cite{Ar1}. When $\of$ consists of the unipotent elements of
$G(\Q)$, we write $J_{\unip}(f)$ for $J_{\of}(f)$. 

Fix $R\ge 1$ and recall the definition of $\varphi:=\varphi_R$ from~\eqref{phiR}. Put
\begin{equation}\label{cptsupp1}
\widetilde h_t^{\tau,p}:=\varphi h_t^{\tau,p}.
\end{equation}
\begin{lem}\label{lem-geounip}
There exists $N_0\in\N$ such that
\[
J_{\geo}(\widetilde h_t^{\tau,p}\otimes \chi_{K(N)})=J_{\unip}(\widetilde h_t^{\tau,p}\otimes \chi_{K(N)})
\]
for all $N\ge N_0$. 
\end{lem}
\begin{proof}
By definition, the support of $\widetilde h^{\tau,p}_t$ is contained in $B_R$. 
Then the support of $\widetilde h_t^{\tau,p}\otimes \chi_{K(N)}$ is contained in
$B_RK(N)\subset B_R\K$, and therefore there are only finitely many classes
$\of\in\cO_G$ that contribute to the geometric side of the trace formula 
\eqref{geo-side1} for the the functions $\widetilde h_t^{\tau,p}\otimes 
\chi_{K(N)}$. Moreover, the only class $\of\in\cO_G$ for which the union of the
$G(\A)$-conjugacy classes of elements of $\of$ meets $G(\R)K(N)$ for 
infinitely many $N\in\N$ is the unipotent class. For assume that $\of$ has
this property. Let $\gamma\in\of$ and let $q\in\Q[X]$ be the characteristic
polynomial of the linear map $\gamma-\Id\in\End(\C^n)$. The assumption on 
$\of$ implies that every coefficient of $q$, except the leading coefficient
$1$, is either arbitrarily close to $0$ at some prime $p$ or has absolute
value $<1$ at infinitely many places. Therefore, necessarily, $q=X^n$, and
$\gamma$ is unipotent. Therefore, the geometric side reduces to 
$J_{\unip}(\widetilde h_t^{\tau,p}\otimes \chi_{K(N)})$ for all but finitely many
$N\in\N$.
\end{proof}

To analyze $J_{\unip}(f)$ we use Arthur's fine geometric expansion 
\cite[Corollaries 8.3]{Ar4} to
express $J_{\unip}(f)$ in terms of weighted orbital integrals. To state the 
result we recall some facts about weighted orbital integrals. Let $S$ be a 
finite set of places of $\Q$ containing $\infty$. Set
\[
\Q_S=\prod_{v\in S}\Q_v,\quad \mathrm{and}\quad G(\Q_S)=\prod_{v\in S}G(\Q_v).
\]
Let $M\in\levis$ and $\gamma\in M(\Q_S)$. 
The general weighted orbital integrals $J_M(\gamma,f)$ defined in 
\cite{Ar5} are distributions on $G(\Q_S)$. If $\gamma$ is such that
$M_\gamma=G_\gamma$, then, as the name suggests,  $J_M(\gamma,f)$ is given by an
integral of the form
\[
J_M(\gamma,f)=\big|D(\gamma)\big|^{1/2}\int_{{G_\gamma}(\Q_S)\bs G(\Q_S)}
f(x^{-1}\gamma x) v_M(x)\ dx,
\]
where $D(\gamma)$ is the discriminant of $\gamma$ \cite[p. 231]{Ar5}
and $v_M(x)$ is the weight function associated to the $(G,M)$-family
$\{v_P(\lambda,x)\colon P\in\mathcal{P}(M)\}$ defined in 
\cite[p.230]{Ar5}. For general $\gamma$ the definition is more 
complicated. In this case, $J_M(\gamma,f)$ is obtained as a limit of a linear
combination of integrals as above. For more details we refer to 
\cite{Ar8}.
Let
 \[
G(\Q_S)^1=G(\Q_S)\cap G(\A)^1
\]
and write $C^\infty_c(G(\Q_S)^1)$ for the space of functions on $G(\Q_S)^1$
obtained by restriction of functions in $C^\infty_c(G(\Q_S))$. If $\gamma$ 
belongs to the intersection of $M(\Q_S)$ with $G(\Q_S)^1$, one can obviously
define the corresponding weighted orbital integral as linear form on 
$C^\infty_c(G(\Q_S)^1)$.

Since for $\GL(n)$ all conjugacy classes are stable (in the sense that for any 
finite set $S$, two unipotent elements in $G(\Q)$ are conjugate in $G(\Q_S)$ 
if and only if they are conjugate in $G(\Q)$), 
the expression of $J_{\unip}(f)$ in terms of weighted 
orbital integrals  simplifies. For 
$M\in\levis$ let $\left(\cU_M(\Q)\right)$ be the (finite) set of unipotent 
conjugacy classes of $M(\Q)$. Let $F\in C^\infty_c(G(\Q_S)^1)$ and denote by
$\1_{K^S}$ the characteristic function of the standard maximal compact subgroup
of $G(\A^S)$. Then by 
\cite[Corollary 8.3]{Ar4} there exist constants $a(S,\OOO)$ which
depend on the normalization of measures such that
\begin{equation}\label{fine-exp}
J_{\unip}(F\otimes\1_{K^S})=\vol(G(\Q)\bs G(\A)^1)F(1)
+\sum_{(M,\OOO)\neq(G,\{1\})}a^M(S,\OOO)J_M(\OOO,F),
\end{equation}
where $M$ runs over $\cL$ and $\OOO$ over $\left(\cU_M(\Q)\right)$.
To deal with the $S$-adic integral, we note that  $J_M(\OOO,F)$ can be 
decomposed  into a sum  of
products of integrals at $\infty$ and at the finite places $S_{f}=S\setminus
\{\infty\}$. Suppose that $F=F_\infty\otimes F_f=F_\infty\otimes\bigotimes_{p\in S_f} F_p$ with $F_v\in C^\infty(G(\Q_v))$.  Let $L\in\cL(M)$ and $Q=LV\in\cP(L)$. Define 
\begin{equation}\label{ct}
F_{\infty,Q}(m)=\delta_Q(m)^{1/2}\int_{K_\infty}\int_{V(\R)}F_\infty(k^{-1}mvk)dkdv,\quad
m\in M(\R),
\end{equation}
and define $F_{f,Q}$ in a similar way. Then for every pair of Levi subgroups
$L_1,L_2\in\cL(M)$ there exist constants $d^G_M(L_1,L_2)\in\C$ such that
\begin{equation}\label{decomp1}
J_M(\cO,F)=\sum_{L_1,L_2\in\cL(M)}d^G_M(L_1,L_2) J^{L_1}_M(\OOO_\infty,F_{\infty,Q_1})
J^{L_2}_M(\OOO_{f},F_{f,Q_2})
\end{equation}
(see \cite{Ar3},\cite[(18.7)]{Ar10}) where $Q_i\in\cP(L_i)$, and
$\OOO_{f}=(\OOO_v)_{v\in S_{f}}$, where for each $v\in S$, $\OOO_v\subseteq M(\Q_v)$ denotes the  $M(\Q_v)$-conjugacy class of $\OOO$.
The coefficients $d_M^G(L_1, L_2)$ are independent of $S$ and they vanish 
unless the natural map $\ka_M^{L_1}\oplus\ka_M^{L_2}\longrightarrow \ka_M^G$ is 
an isomorphism. In case the coefficient does not vanish, it depends on the 
chosen measures on $\ka_M^{L_1}$, $\ka_M^{L_2}$ and $\ka_M^G$. 

\begin{lem}\label{lem:vanishing:coeff}
 If $d_M^G(\underline{L})\neq0$, then at most $\dim\aaa_M^G$-many elements of $\underline{L}$ are not equal to $M$.
\end{lem}
\begin{proof}
 The first assertion is clear from the fact that the map in \eqref{eq:isom} is an isomorphism if $d_M^G(\underline{L})\neq0$. 
\end{proof}

We shall apply \eqref{fine-exp} and \eqref{decomp1} with test functions 
$F$ satisfying $F_f=\1_{K(N)}$. In this case we can choose the set of 
places $S=S(N)$ quite explicitly and also have a good upper bound for the 
global coefficients $a^M(S(N),\OOO)$ that occur in \eqref{fine-exp}. Namely we
have 
\begin{lem}\label{lem-glob-coeff}
\begin{enumerate}
\item Let $S(N)=\{\infty\}\cup \{p: ~ p|N\}$. Then \eqref{fine-exp} with $S=S(N)$ holds for $F=F_\infty\otimes\One_{K(N)}$.  
\item There exist constants $a,b>0$ such that for all $N$, all $M$ and all 
unipotent orbits $\OOO$ in $M$ we have
  \[
   |a^M(S(N),\OOO)| \le a (1+\log N)^b 
  \]
with $S(N)$ as in the first part.
\end{enumerate}
\end{lem}
\begin{proof}
The first statement is contained in \cite[Corollary 8.3]{Ar4}. The second statement follows from \cite{coeff}, see also \cite[\S 6]{Ma}. 
\end{proof}

In the following we write
\[
 N=\prod_p p^{e_p}
\]
for the prime factorization of $N$. Then $\One_{K(N)}=\bigotimes_{p} \One_{K(p^{e_p})}$ with $K(p^{e_p})$ the principal congruence subgroup of level $p^{e_p}$ in $\cpt_p=\GL_n(\Z_p)$.

We can assume that $L_2=G$ since $L_2$ is canonically isomorphic to a direct product of smaller $\GL(m)$s.
We then
split the finite orbital integral $J^G_M(\OOO_{f},\One_{K(N)})$ further, until we arrive at
\begin{equation}\label{decomp4}
J^G_M(\OOO_{f},\One_{K(N)})
=\sum_{\underline{L}\in\cL(M)^{|S(N)_f|}}d_M^G(\underline{L})
\prod_{p\in S(N)_f}J_M^{L_p}(\OOO_p,\One_{\One_{K(p^{e_p})},Q_p}),
\end{equation}
where $\underline{L}$ runs over all tuples $(L_p)_{p\in S(N)_f}$
of Levi subgroups $L_p\in\cL(M)$,  and $d_M^G(\underline{L})$ are
certain constants satisfying $d_M^G(\underline{L})=0$ unless the natural map
\begin{equation}\label{eq:isom}
\bigoplus_{p\in S(N)_f}\af_0^{L_p}\to \af_0^G
\end{equation}
is an isomorphism. Moreover, the parabolic subgroups $Q_p\in\cP(L_p)$ are unique and
chosen as explained in \cite[\S 17-18]{Ar10}.

It follows from \cite{Ar5} (see also \cite{LM}) that each 
local integral can be written as (using that $K(p^{e_p})$ is normal in $K_p$)
\begin{equation}\label{local-int}
J_M^{L_p}(\OOO_p,\One_{K(p^{e_p}),Q_p})=\int_{N_p(\Q_p)} \One_{K(p^{e_p}),Q_p}(n)w^{L_p}_{M,\OOO_p}(n)\,dn,
\end{equation}
where $P_p=M_pN_p\subset L_v$ is a standard parabolic subgroup with 
$M_p\subset M$ such that $\OOO_p$ is induced from the trivial orbit in 
$M_p$ to $M$, i.e., $P_p$ is a Richardson parabolic for $\OOO_p$ in $M$. The 
function $w^{L_p}_{M,\OOO_p}$ is a certain weight function on $N_p(\Q_p)$ of the form
\begin{equation}\label{weight-fct}
w^{L_p}_{M,\OOO_p}=Q(\log\|q_1(X)\|_p,\dots,\log\|q_r(X)\|_p),
\end{equation}
where $n=\Id+X$ with $X$ a nilpotent upper triangular matrix, $q_1,\dots,q_r$
are polynomials in $X$ with image in some affine space, and $Q$ is a polynomial.
Note that $Q,q_1,\dots,q_r$ only depend on $\OOO$, $M$, and $L_p$ (as a Levi
subgroup of $G$ defined over $\Q$), but not on the place~$p$.

\section{Bounds for $p$-adic orbital integrals}\label{sec-padic}
\setcounter{equation}{0}

In this section we still assume that $G=\GL(n)$. We deal with the orbital integrals 
of the form $J_M^L(\OOO,\1_{K(N),Q})$, $Q\in\cP(L)$, which arise in
\eqref{decomp1} for our type of test functions. 

We first make the following observation:
Let $Q=LV$ be a semistandard parabolic subgroup.
Since $K(N)\cap V(\A_f)= V(N\hat{\Z})$, we have
\begin{equation}\label{eq:int} 
 \int_{V(\A_f)} \One_{K(N)}(v)\, dv
 = N^{-\dim V}.
\end{equation}
By the definition \eqref{ct} and the fact that $K(N)$ is a normal subgroup in 
$\cpt_f$ we have
\[
 \One_{K(N), Q}(m) = \delta_Q(m)^{1/2}\int_{V(\A_f)} \One_{K(N)} (mv)\, dv
\]
for any  $m\in L(\A_f)$. Hence $\One_{K(N),Q}(m)=0$ unless $m\in K^L(N)=K(N)\cap L(\A_f)$. Now if $m\in K^L(N)$, we have $mv\in K(N)$ if and only if $v\in K(N)$. Hence
\begin{equation}\label{eq:descent}
 \One_{K(N), Q}(m)= N^{-\dim V} \One_{K^L(N)}(m). 
\end{equation}
It therefore suffices to bound $J_M^{L}(\OOO,\One_{K^{L}(N)})$. 
Again, since $L$ is isomorphic to a direct product of finitely many smaller 
$\GL(m)$'s, it suffices to consider the case $Q_2=G=\GL(n)$.
Moreover, the formulas similar to \eqref{eq:int} and \eqref{eq:descent} hold for the local integrals at $p$ for the functions $\One_{K(p^{e_p})}$ with the necessary adjustments.

We now use \eqref{decomp4} to find an upper bound for the orbital integrals.

\begin{lem}\label{lem:unweighted}

If $L_p=M$, then 
\begin{equation}\label{eq:descent:integral}
 J_M^{L_p}(\OOO_p, \One_{K(p^{e_p}),Q_p})
= J_M^M(\OOO_p, \One_{K(p^{e_p}),Q_p}) 
=  p^{-\frac{e_p}{2}\dim \Ind_M^G\OOO} 
 \end{equation}
\end{lem}
\begin{proof}
 Let $Q_p=M V$ be the Iwasawa decomposition of $Q_p$ and let $P^M= L^M U^M$ be a Richardson parabolic in $M$ for $\OOO_p$ with $T_0\subseteq L^M$, that is, $\OOO$ is induced from the trivial orbit in $L^M$ to $M$. Then $L^M U^M V=:L^M U^G$ is a Richardson parabolic for the induced orbit $\Ind_M^G\OOO_p$. Since $K(p^{e_p})$ is a normal subgroup in $K_p$, we can compute the invariant orbital integral $J_M^M(\OOO, \One_{K(p^{e_p}),Q_p})$ as (cf.\ also \eqref{local-int} and \cite{LM})
 \begin{multline*}
J_M^M(\OOO, \One_{K(p^{e_p}),Q_p})
=  \int_{U^M(\Q_p)} \One_{K(p^{e_p}),Q_p}(u)\, du
  = \int_{U^M(\Q_p)} \int_{V(\Q_p)} \One_{K(p^{e_p})}(uv)\,dv\, du\\
 = \int_{U^G(\Q_p)}  \One_{K(p^{e_p})}(u)\, du.
 \end{multline*}
 Since $\dim U^G=\dim\Ind_M^G\OOO /2$, the equation \eqref{eq:descent:integral} follows from \eqref{eq:int}.
\end{proof}

Recall from \eqref{local-int} and \eqref{weight-fct} that
\[
 J_M^{L_p}(\OOO, \One_{K(p^{e_p}),Q_p})
 =\int_{N_p(\Q_p)} \One_{K(p^{e_p}),Q_p}(n) w^{L_p}_{M,\OOO}(n)\, dn
\]
The polynomials $Q, q_1, \ldots, q_r$ defining $w^{L_p}_{M,\OOO}$ only depend on $\OOO$, $M$, and $L_p$ (as a Levi subgroup of $G$ defined over $\Q$), but not on the prime $p$. Hence there are overall only finitely many possibilities for those polynomials independent of the level $N$. Now if $n\in K(p^{e_p})\cap N_p(\Q_p)$ we can write $n=\Id + p^{e_p} Y$ with $Y\in \Mat_{n\times n}(\Z_p)$ a nilpotent matrix. Hence setting $n'=\Id +Y$ we get
\begin{multline*}
\left|  J_M^{L_p}(\OOO, \One_{K(p^{e_p}),Q_p})\right|
 \le p^{-e_p\dim V_p}\int_{N_p(\Q_p)} \One_{K^{L_p}(p^{e_p})}(n) |w^{L_p}_{M,\OOO}(n)|\, dn\\
 \le p^{-e_p\dim V_p} p^{-e_p\dim N_p} \int_{N_p(\Q_p)} \One_{K^{L_p}_p}(n') Q'(\log p^{e_p}, |\log \|q_1(Y)\|_p|, \ldots, |\log \|q_r(Y)\|_p|)\, dn' 
\end{multline*}
with $Q'$ a suitable polynomial only depending on $Q, q_1, \ldots, q_r$ and $n$ but not on $N$.

\begin{lem}
There exist absolute constants $r, p>0$ (independent of $p, N$) such that 
\[
 \int_{N_p(\Q_p)} \One_{K^{L_p}_p}(n') Q'(\log p^{e_p}, |\log \|q_1(Y)\|_p|, \ldots, |\log \|q_p(Y)\|_p|)\, dn' 
 \le C (1+\log p^{e_p})^r.
\]
\end{lem}
\begin{proof}
There exists another polynomial $\tilde Q$ and some integer $j>0$ such that
\[
\begin{split}
 Q'(\log p^{e_p}, |\log \|q_1(Y)\|_p|,& \ldots, |\log \|q_p(Y)\|_p|)\\
 &\le (1+\log p^{e_p})^j   
\tilde Q(|\log \|q_1(Y)\|_p|, \ldots, |\log \|q_p(Y)\|_p|)
\end{split}
\]
for all $n'=\Id+ Y$.  We can assume that $\tilde Q$ is independent of $p$ and does only depend on $Q'$. But now by \cite[\S 10]{Ma} there exists a constant $C>0$ such that 
\[
  \int_{N_p(\Q_p)} \One_{K^{L_p}_p}(n') \tilde Q(\log p^{e_p}, |\log \|q_1(Y)\|_p|, \ldots, |\log \|q_p(Y)\|_p|)\, dn'
\le C
  \]
and $C$ can be chosen to depend only on $\tilde Q$ and $n$ but not on $p$.
\end{proof}

Together with the discussion previous to the lemma this immediately implies the following:
\begin{cor} 
With the notation as before, we have 
\[
 \left|  J_M^{L_p}(\OOO, \One_{K(p^{e_p}),Q_p})\right|
 \le C p^{-\frac{e_p}{2} \Ind_M^G \OOO} (1+ \log p^{e_p})^r
\]
with $r$ and $C$ chosen to depend only on $n$ but not on $p$ or $N$.
\end{cor}
\begin{proof}
It remains to note that $\dim V_p+\dim N_p$ equals half the dimension of the induced class $\Ind_M^G \OOO$ see \cite[Theorem 7.1.1]{CM}.
\end{proof}

The estimate in the corollary can also be written as
\[
 \left|  J_M^{L_p}(\OOO, \One_{K(p^{e_p}),Q_p})\right|
\le  C |N|_p^{\dim \Ind_M^G\OOO} (1 -\log |N|_p)^r.
\]
Combining this with Lemma \ref{lem:unweighted} we get
\[
 \left|  J_M^{L_p}(\OOO, \One_{K(p^{e_p}),Q_p})\right|
 \begin{cases}
  \le C |N|_p^{\dim \Ind_M^G\OOO} (1 -\log |N|_p)^r &\text{if } L_p\neq M,\\
  =  |N|_p^{\dim \Ind_M ^G\OOO} &\text{if } L_p=M.
 \end{cases}
\]
By Lemma \ref{lem:vanishing:coeff} we have for any tuple $\underline{L}=\{L_p\}_{p\in S(N)_f}$ with $d_M^G(\underline{L})\neq0$ that
\begin{multline*}
 \left| \prod_{p\in S(N)_f} J_M^{L_p}(\OOO, \One_{K^{L_p}(p^{e_p}),Q_p})\right|
 \le N^{-\dim \Ind_M^G\OOO/2} C^{\dim\aaa_M^G} \prod_{p\in S(N)_f:L_p\neq M} (1- \log|N|_p )^r\\
 \le c N^{-\dim \Ind_M^G\OOO/2} (\log N)^{r(n-1)}
\end{multline*}
for some absolute constant $c>0$ independent of $N$.
 Lemma \ref{lem:vanishing:coeff} also implies that the number of tuples $\underline{L}$ with $d_M^G(\underline{L})\neq0$ is bounded by $|S(N)_f|^{\dim\aaa_M^G}$. Since the number of elements in $S(N)_f$ is equal to the number $\omega(N)$ of prime factors of $N$, and $\omega(N)\le\log_2N\le 2\log N$, we get that  
for any $N\ge 2$ we have  
\begin{equation}\label{eq:finiteest}
\left|  J_M^{L_2}(\OOO,\One_{K(N),Q_2})\right|
\le c' N^{-\dim \Ind_M^G\OOO/2} (\log N)^{(r+1)(n-1)}
\end{equation}
for some absolute constant $c'>0$.

\section{Proof of the main result for $\GL(n)$}
\label{sec-main-gln}
\setcounter{equation}{0}

Let $G=\GL(n)$. Let $K(N)\subset \GL(n,\A_f)$ be the principal congruence 
subgroup and 
\[
Y(N):=X(K(N))
\]
the associated adelic quotient 
\eqref{adelic-quot2}. 
Let $\tau\in\Rep(G(\R)^1)$ satisfying $\tau\ncong\tau_\theta$. Let $E_\tau$
be the associated flat vector bundle over $Y(N)$ as defined in section 
\ref{sec-analtor}.
Let $\Delta_{p,Y(N)}(\tau)$ be the Laplace operator on $E_\tau$-valued $p$-forms on
$Y(N)$. For $t>0$ let $e^{-t\Delta_{p,Y(N)}(\tau)}$ be the heat operator. The
regularized trace $\Tr_{\reg}(e^{-t\Delta_{p,Y(N)}(\tau)})$ of the heat operator 
$e^{-t\Delta_{p,Y(N)}(\tau)}$ is defined by \eqref{reg-trace7}. By 
\eqref{large-time7} 
and \eqref{small-time5} the zeta function $\zeta_{p,N}(s;\tau)$ is defined by
 \begin{equation}\label{zetafct}
\zeta_{p,N}(s;\tau):=\frac{1}{\Gamma(s)}\int_0^\infty 
\Tr_{\reg}\left(e^{-t\Delta_{p,Y(N)}(\tau)}\right)t^{s-1} dt.
\end{equation}
The integral converges absolutely and uniformly on compact subsets of the 
half-plane $\Re(s)>d/2$, and admits a meromorphic extension to the entire
complex plane. Then the analytic torsion $T_{Y(N)}(\tau)\in\R^+$ is defined by
\begin{equation}\label{anator2}
\log T_{Y(N)}(\tau)=\frac{1}{2}\sum_{p=0}^d (-1)^p p \left(\FP_{s=0}
\frac{\zeta_{p,Y(N)}(s;\tau)}{s}\right)
\end{equation}
(see \cite[(13.38)]{MzM}). Let $T>0$. We write
\begin{equation}\label{splitt-at}
\begin{split}
\int_0^\infty \Tr_{\reg}\left(e^{-t\Delta_{p,Y(N)}(\tau)}\right)t^{s-1}dt 
&=\int_0^T \Tr_{\reg}\left(e^{-t\Delta_{p,Y(N)}(\tau)}\right)t^{s-1} dt\\
&+\int_T^\infty \Tr_{\reg}\left(e^{-t\Delta_{p,Y(N)}(\tau)}\right)t^{s-1}dt.
\end{split}
\end{equation}
We first deal with the second integral on the right hand side. Note that the 
integral is an entire function of $s$. Therefore, we have
\[
\FP_{s=0}\left(\frac{1}{s\Gamma(s)}\int_T^\infty 
\Tr_{\reg}\left(e^{-t\Delta_{p,Y(N)}(\tau)}\right)t^{s-1}dt\right)=
\int_T^\infty \Tr_{\reg}\left(e^{-t\Delta_{p,Y(N)}(\tau)}\right)t^{-1}dt.
\]
Using Proposition \ref{prop-larget} it follows that there exist $C,c>0$ such
that
\begin{equation}\label{est-larget3}
\frac{1}{\vol(Y(N))}\left|\int_T^\infty 
\Tr_{\reg}\left(e^{-t\Delta_{p,Y(N)}(\tau)}\right)t^{-1}dt\right|\le C e^{-cT}
\end{equation}
for all $T\ge 1$, $p=0,\dots,d$, and $N\in\N$. 

Now we turn to the first integral on the right hand side of \eqref{splitt-at}.
Recall that 
\[
\Tr_{\reg}\left(e^{-t\Delta_{p,Y(N)}(\tau)}\right)=J_{\spec}(h_t^{\tau,p}\otimes\chi_{K(N)}).
\]
For $R>0$ let $\varphi_R\in C^\infty_c(G(\R)^1)$ be the function defined by
\eqref{phiR}. By
Proposition \ref{prop-cptsupp} we have
\begin{equation}\label{cptsupp2}
\Tr_{\reg}\left(e^{-t\Delta_{p,Y(N)}(\tau)}\right)=
J_{\spec}(\varphi_R h_t^{\tau,p}\otimes\chi_{K(N)})+r_R(t),
\end{equation}
where $r_R(t)$ is a function of $t\in [0,T]$ which satisfies
\begin{equation}\label{r-estim}
\frac{1}{\vol(Y(N))}|r_R(t)|\le C_1 e^{-C_2R^2/t+C_3t}
\end{equation}
for $0\le t\le T$. This implies that $\int_0^T r_R(t) t^{s-1}dt$ is holomorphic 
in $s\in\C$ and 
\[
FP_{s=0}\left(\frac{1}{s\Gamma(s)} \int_0^Tr_R(t) t^{s-1}dt\right)= 
\int_0^Tr_R(t) t^{-1}dt.
\]
Moreover
\begin{equation}\label{est-remain}
\begin{split}
\frac{1}{\vol(Y(N))}\left|\int_0^T r_R(t)t^{-1}dt\right|&\le 
C_1\int_0^T e^{-C_2R^2/t+C_3t} t^{-1} dt\\
&\le C_1 e^{-C_4R^2/T+C_3T} \int_0^{T/R^2} e^{-C_4/t}t^{-1} dt.
\end{split}
\end{equation}
Now put $R=T^2$ and let 
\begin{equation}\label{cpt-supp4}
h^{\tau,p}_{t,T}:=\varphi_{T^2} h_t^{\tau,p}.
\end{equation}
Then it 
follows from \eqref{cptsupp2} and \eqref{est-remain}
that there exist $C,c>0$ such that
\begin{equation}\label{cpt-supp3}
\begin{split}
\frac{1}{\vol(Y(N))}\biggl|\FP_{s=0}&\left(\frac{1}{s\Gamma(s)}\int_0^T
\Tr_{\reg}\left(e^{-t\Delta_{p,Y(N)}(\tau)}\right)t^{s-1}dt\right)\\
&-\FP_{s=0}\left(\frac{1}{s\Gamma(s)}
\int_0^TJ_{\spec}(h_{t,T}^{\tau,p}\otimes\chi_{K(N)})t^{s-1}dt\right)\biggr|
\le C e^{-cT}
\end{split}
\end{equation}
for $T\ge 1$, $p=0,\dots,d$, and $N\in\N$. Using the trace formula, we are 
reduced to deal with
\[
\FP_{s=0}\left(\frac{1}{s\Gamma(s)}
\int_0^TJ_{\geo}(h_{t,T}^{\tau,p}\otimes\chi_{K(N)})t^{s-1}dt\right).
\]
Let $\varphi\in C_c^\infty(G(\R)^1)$ be such that $\varphi(g)=1$ in a 
neighborhood of $1\in G(\R)^1$. Put 
\[
\widetilde h_t^{\tau,p}=\varphi h_t^{\tau,p}.
\]
We consider test functions with $\widetilde h_t^{\tau,p}$ at the infinite place.
By Lemma \ref{lem-geounip} there exists $N_0\in\N$ such that
\[
J_{\geo}(\widetilde h_t^{\tau,p}\otimes\chi_{K(N)})
=J_{\unip}(\widetilde h_t^{\tau,p}\otimes\chi_{K(N)})
\]
for $N\ge N_0$.  Let $S(N)$ be as in Lemma \ref{lem-glob-coeff}. By the fine geometric expansion \eqref{fine-exp} and the 
definition of $h^{\tau,p}_{t,T}$ we have
\begin{equation}\label{fine-exp4}
\begin{split}
J_{\unip}(\widetilde h_t^{\tau,p}\otimes\chi_{K(N)})=&\vol(G(\Q)\bs G(\A)^1/K(N))
\widetilde h_t^{\tau,p}(1)\\
&+\sum_{(M,\OOO)\neq(G,\{1\})}a^M(S(N),\OOO)
J_M(\OOO,\widetilde h_t^{\tau,p}\otimes \chi_{K(N)}).
\end{split}
\end{equation}
Concerning the volume factor in the first summand, we used that $\chi_{K(N)}=\1_{K(N)}/\vol(K(N))$. 
To begin with we consider the first term on the right hand side. Note that
$\widetilde h_t^{\tau,p}(1)=h_t^{\tau,p}(1)$. Furthermore, by
\cite[(5.11)]{MP2} there is an asymptotic expansion
\begin{equation}\label{asymp-exp}
h_t^{\tau,p}(1)\sim \sum_{j=0}^\infty a_j t^{-d/2+j}
\end{equation}
as $t\to 0$. Furthermore, by \cite[(5.16)]{MP2} there exists $c>0$ such that
\begin{equation}\label{larget3}
h_t^{\tau,p}(1)=O(e^{-ct})
\end{equation}
as $t\to\infty$. From \eqref{asymp-exp} and \eqref{larget3} follows that the 
integral
\begin{equation}
\int_0^\infty h_t^{\tau,p}(1) t^{s-1}dt
\end{equation}
converges in the half-plane $\Re(s)>d/2$ and admits a meromorphic extension
to $\C$ which is holomorphic at $s=0$. The same is true for the integral 
over $[0,T]$ and we get
\begin{equation}\label{identcontr}
\FP_{s=0}\left(\frac{1}{s\Gamma(s)}\int_0^T h_t^{\tau,p}(1) t^{s-1}\right)=
\frac{d}{ds}\left(\frac{1}{\Gamma(s)}\int_0^\infty h_t^{\tau,p}(1) t^{s-1}\right)
\bigg|_{s=0}+O(e^{-cT}).
\end{equation}
Recall the definition of the $L^{(2)}$-analytic torsion \cite{Lo}, \cite{MV}.
For $t>0$ let 
\[
K^{(2)}(t,\tau):=\sum_{p=1}^d (-1)^p p h_t^{\tau,p}(1).
\]
Put
\[
t^{(2)}_{\widetilde X}(\tau):=\frac{1}{2}\frac{d}{ds}\left(\frac{1}{\Gamma(s)}
\int_0^\infty K^{(2)}(t,\tau) t^{s-1} dt\right)\bigg|_{s=0}.
\]
Then by \cite[(5.20)]{MP2}, the $L^{(2)}$-analytic torsion $
T^{(2)}_{Y(N)}(\tau)\in\R^+$ is given by
\[
\log T^{(2)}_{Y(N)}(\tau)=\vol(Y(N))\cdot t^{(2)}_{\widetilde X}(\tau).
\] 
To summarize, we get
\begin{equation}\label{l2-tor1}
\frac{1}{2}\sum_{p=1}^d (-1)^p p \FP_{s=0}\left(\frac{1}{s\Gamma(s)}
\int_0^T h_t^{\tau,p}(1) t^{s-1} dt\right)=t^{(2)}_{\widetilde X}(\tau)+ O(e^{-cT})
\end{equation}
for $T\ge 1$.

Next we consider the weighted orbital integrals on the right hand side of
\eqref{fine-exp4}. Note that by definition of $\chi_{K(N)}$ we have
\[
J_M(\OOO,\widetilde h^{\tau,p}_t\otimes \chi_{K(N)})=\frac{1}{\vol(K(N))}
J_M(\OOO,\widetilde h^{\tau,p}_t\otimes \1_{K(N)}).
\]
To deal with the integral on the right hand side, we use the decomposition
formula \eqref{decomp1}. For $L\in\cL(M)$, $Q\in\cP(L)$, and a unipotent conjugacy class $\OOO$ in $M(\Q)$ consider the integral $J^L_M(\OOO,(\widetilde h_t^{\tau,p})_Q)$.
Unfolding the definition $(\tilde{h}_t^{\tau,p})_{Q}$, the local weighted 
orbital integral $ J_M^{L}((\OOO,\tilde{h}_t^{\tau,p})_{Q})$ can be written as 
a non-invariant integral over the unipotent radical of a suitable  
semistandard parabolic subgroup in $G$. More precisely, there is a 
semistandard parabolic subgroup $R=M_RU_R\subseteq M$ which is a Richardson 
parabolic for $\OOO$ in $M$. If $Q=LV$ is the Levi decomposition of 
$Q$, we get
\[
J_M^{L}(\OOO,(\tilde{h}_t^{\tau,p})_{Q})
=\int_{V(\R)}\int_{U_R(\R)}  \tilde{h}_t^{\tau,p}(uv) w(u)\, du\, dv
\]
where $w$ is a certain weight function depending on the class $\OOO$, and the 
groups $M$ and $L$. This weight function on $U_R(\R)$ satisfies a certain 
``log-homogeneity'' property as explained in \cite[\S 6-7]{MzM}. Note that 
$M_RU_RV=:M_R V'$ is a Richardson parabolic for the induced class 
$\Ind_M^G\OOO$ in $G$. Extending $w$ trivially to all of $V'(\R)$ (and writing 
$w$ for the extension again), we get
\[
 J_M^{L}(\OOO,(\tilde{h}_t^{\tau,p})_{Q})
=\int_{V'(\R)} \tilde{h}_t^{\tau,p}(v) w(v)\, dv
\]
and this extended $w$ is again log-homogeneous. It follows from 
\cite[\S 12]{MzM} that this integral admits an asymptotic expansion as
$t\to 0$. This implies that the integral
\begin{equation}\label{partial-mellin}
\int_0^T J_M^L(\OOO,(\tilde{h}_t^{\tau,p})_{Q}) t^{s-1}dt
\end{equation}
converges absolutely and uniformly on compact subsets of $\Re(s)>d/2$ and
admits a meromorphic extension to $s\in\C$. Put
\begin{equation}\label{fp-orbint}
A_M^L(\OOO_\infty,T):=\FP_{s=0}\left(\frac{1}{s\Gamma(s)}
\int_0^T J_M^L(\OOO,(\widetilde h_t^{\tau,p})_Q) t^{s-1} dt\right).
\end{equation}
By \eqref{decomp1} it follows that the Mellin transform of 
$J_M(\OOO,\widetilde h_t^{\tau,p}\otimes\1_{K(N)})$ as a function of $t$ is a
meromorphic function on $\C$, and we get
\begin{equation}\label{decomp2}
\begin{split}
\FP_{s=0}\biggl(\frac{1}{s\Gamma(s)}&\int_0^T 
J_M(\OOO,\widetilde h_t^{\tau,p}\otimes\1_{K(N)})t^{s-1} dt\biggr)\\
&= \sum_{L_1,L_2\in\cL(M)} d_M^G(L_1,L_2) 
A_M^{L_1}(\OOO_\infty,T) J_M^{L_2}(\OOO_f,\1_{K(N),Q_2}).
\end{split}
\end{equation}

Denote by $J_{\unip-\{1\}}(\widetilde h_t^{\tau,p}\otimes\1_{K(N)})$ the 
sum on the right hand side of \eqref{fine-exp4} with the term
$\vol(G(\Q)\backslash G(\A)^1/K(N)) \widetilde h_t^{\tau, p}(1)$ removed. 
Combining \eqref{decomp2}, Lemma \ref{lem-glob-coeff},
and \eqref{eq:finiteest}, we obtain

\begin{prop}\label{prop-unip}
For every $T\ge 1$  there exist  constants $C(T),a>0$, $a$ 
independent of $T$, such that for all $N\ge 2$ we have
\[
\biggl| \FP_{s=0}\biggl(\frac{1}{s\Gamma(s)}\int_0^T 
J_{\text{unip} - \{1\}} (\tilde{h}_t^{\tau,p}\otimes \1_{K(N)}) t^{s-1}\, 
dt\biggr)\biggr|\le C(T) N^{-(n-1)} (\log N)^a.
\]
\end{prop}
Now we can turn to the proof of the main Theorem. Let
\[
K_N(t,\tau):=\frac{1}{2}\sum_{p=1}^d (-1)^p p 
\Tr_{\reg}\left(e^{-t\Delta_{p,Y(N)}(\tau)}\right).
\]

Let $T>0$. By \eqref{zetafct}, \eqref{anator2}
and \eqref{splitt-at} we have 
\begin{equation}\label{anator3}
\log T_{Y(N)}(\tau)=\FP_{s=0}\left(\frac{1}{s\Gamma(s)}
\int_0^T K_N(t,\tau) t^{s-1}dt\right)+\int_T^\infty K_N(t,\tau) t^{-1}dt.
\end{equation}
By \eqref{est-larget3} there exist $C,c>0$ such that 
\begin{equation}\label{est-larget4}
\frac{1}{\vol(Y(N))}\left|\int_T^\infty K_N(t,\tau) t^{-1}dt\right|\le C e^{-cT}
\end{equation}
for all $T\ge 1$ and $N\in\N$. Let $h_{t,T}^{\tau,p}\in C_c^\infty(G(\R)^1)$
be defined by \eqref{cpt-supp4}. Put
\[
K_N(t,\tau;T):=\frac{1}{2}\sum_{p=1}^d (-1)^p p 
J_{\geo}(h_{t,T}^{\tau,p}\otimes \chi_{K(N)}).
\]
By \eqref{cpt-supp3} and the trace formula it follows that there exist 
$C,c>0$ such that
\begin{equation}\label{cpt-supp6}
\begin{split}
\frac{1}{\vol(Y(N))}\biggl|\FP_{s=0}\biggl(\frac{1}{s\Gamma(s)}&
\int_0^T K_N(t,\tau) t^{s-1}dt\biggr)\\
&-\FP_{s=0}\left(\frac{1}{s\Gamma(s)}
\int_0^T K_N(t,\tau;T) t^{s-1}dt\right)\biggr|\le C e^{-cT}
\end{split}
\end{equation}
for all $T\ge 1$ and $N\in\N$. Let
\begin{equation}\label{cpt-supp5}
K_{\unip-\{1\},N}(t,\tau;T):=\frac{1}{2}\sum_{p=1}^d (-1)^p p 
J_{\unip-\{1\}}(h_{t,T}^{\tau,p}\otimes \chi_{K(N)}).
\end{equation}
By Lemma \ref{lem-geounip} and \eqref{fine-exp4} it follows that for every 
$T\ge 1$ there exists $N_0(T)\in\N$ such that 
\[
K_N(t,\tau;T)=\frac{\vol(Y(N))}{2}\sum_{p=1}^d (-1)^p p h_{t,T}^{\tau,p}(1)
+K_{\unip-\{1\},N}(t,\tau;T)
\]
for $N\ge N_0(T)$. Using \eqref{l2-tor1} and \eqref{cpt-supp6} it follows that 
for every $T\ge 1$ there exists $N_0(T)\in\N$ such that 
\begin{equation}\label{eq:l2-tor}
\begin{split}
\frac{1}{\vol(Y(N))}&\FP_{s=0}\biggl(\frac{1}{s\Gamma(s)}
\int_0^T K_N(t,\tau) t^{s-1}dt\biggr)\\
&=t^{(2)}_{\widetilde X}(\tau)+\frac{1}{\vol(Y(N))}\FP_{s=0}
\biggl(\frac{1}{s\Gamma(s)}\int_0^T K_{\unip-\{1\},N}(t,\tau) t^{s-1}dt\biggr)\\
&\hskip10pt+O(e^{-cT}).
\end{split}
\end{equation}
for $N\ge N_0(T)$. Applying Proposition \ref{prop-unip} we get that for
every $T\ge 1$ there exist constants $C_1(T), C_2, a,c>0$ and $N_0(T)\in\N$ 
such that
\begin{equation}\label{est-l2-tor2}
\begin{split}
\biggl|\frac{1}{\vol(Y(N))}\FP_{s=0}\biggl(\frac{1}{s\Gamma(s)}
\int_0^T K_N(t,\tau) &t^{s-1}dt\biggr)-t^{(2)}_{\widetilde X}(\tau)\biggr|\\
&\le C_1(T) N^{-(n-1)}(\log N)^a+C_2 e^{-cT}
\end{split}
\end{equation}
for $N\ge N_0(T)$. Combined with \eqref{anator3} and \eqref{est-larget4}
it follows that
\begin{equation}\label{limittor3}
\lim_{N\to\infty}\frac{\log T_{Y(N)}(\tau)}{\vol(Y(N))}=t^{(2)}_{\widetilde X}(\tau).
\end{equation}

\section{Proof of the main result for $\SL(n)$}
\label{sec-main-sln}
\setcounter{equation}{0}

The following section is due to Werner Hoffmann. 
In order to deduce Theorem \ref{theo-main} from \eqref{limittor3}, we need to
compare the trace formulas for $\GL(n)$ and $\SL(n)$. This is the purpose of 
the current section. Let $K_f\subset \GL(n,\A_f)$
be an open compact subgroup. Then $A_G(\R)^0 \GL(n, \Q)\backslash \GL_n(\A)/K_f$ is a right $\SL(n,\R)$-space with finitely many orbits. Let 
$g_1,\ldots, g_r\in \GL(n, \A_f)$ be representatives for these orbits. Then as 
right $\SL(n, \R)$-spaces we get
\begin{equation}\label{isom9}
 A_G(\R)^0\GL(n, \Q)\backslash \GL_n(\A)/K_f
 \simeq \bigsqcup_{j=1}^r \Gamma_{g_j, K_f} \backslash \SL(n, \R)
\end{equation}
with 
\begin{equation}\label{g-comp}
 \Gamma_{g_j, K_f}:= \SL(n, \R)\cap \left(\GL(n, \Q)\cdot g_j K_f g_j^{-1}\right)\subseteq \SL(n, \R),
\end{equation}
cf.\ \cite[\S 2]{Ar10}.
Accordingly,
\begin{equation}\label{iso-modules1}
L^2(A_G(\R)^0\GL(n,\Q)\bs\GL(n,\A)/K_f)\cong\bigoplus_{j=1}^r L^2(\Gamma_{g_j,K_f}\bs \SL(n,\R)),
\end{equation}

Now note that the right regular representation $R$ of the group $\GL(n,\A)^1$ 
in the Hilbert space 
$L^2(A_G(\R)^0\GL(n,\Q)\bs\GL(n,\A))$ induces a representation of the 
convolution 
algebra $L^1(A_G(\R)^0K_f\bs\GL(n,\A)/K_f)$ in the Hilbert space 
$L^2(A_G(\R)^0\GL(n,\Q)\bs\GL(n,\A)/K_f)$. For 
$h\in L^1(A_G(\R)^0K_f\bs\GL(n,\A)/K_f)$ let
\[
K_h(x,y):=\sum_{\gamma\in\GL(n,\Q)}h(x^{-1}\gamma y).
\]
Then we have
\[
(R(h)\phi)(x)=\int_{A_G(\R)^0\GL(n,\Q)\bs \GL(n,\A)/K_f}K_h(x,y)\phi(y)dy.
\]
With respect to the isomorphism \eqref{isom9} the kernel $K_h$ is given by
the components 
\[
\Gamma_{g_j,K_f}\bs\SL(n,\R)\times \Gamma_{g_k,K_f}\bs\SL(n,\R)\ni (x,y)\mapsto
K_h(g_jx,g_ky).
\]
If $h$ acts on the right hand side of \eqref{iso-modules1} by these integral
kernels, \eqref{iso-modules1} becomes an isomorphism of 
$L^1(A_G(\R)^0 K_f\bs\GL(n,\A)/K_f)$-modules. Especially assume  that 
$h=h_\infty\otimes \chi_{K_f}$.  Then it follows from \eqref{g-comp} that
\[
K_h(gx,gy)=\sum_{\gamma\in\Gamma_{g,K_f}} h_\infty(x^{-1}\gamma y).
\]
Now we turn to the trace formula. We briefly recall the definition of the 
distribution $J^T(f)$, $f\in C_c^\infty(A_G\bs G(\A))$. For details see \cite{Ar1}.
Let $P=M_PN_P$ be a standard 
parabolic subgroup of $G$ and let $Q$ be a parabolic subgroup containing $P$.
Let $\tau_Q^P$ and $\widehat\tau_P^P$ denote the 
characteristic functions of the set
\[
\{X\in\af_0\colon \langle\alpha,X\rangle>0\;\text{for}\;\text{all}\;
\alpha\in\Delta_P^Q\}
\]
and
\[
\{X\in\af_0\colon \langle\varpi,X\rangle>0\;\text{for}\;\text{all}
\;\varpi\in\hat\Delta_P^Q\},
\]
respectively. If $Q=G$, we will suppress the superscript. Moreover we 
put $\tau_0:=\tau_0^G$ and $\hat\tau_0:=\widehat\tau_0^G$. Let
\begin{equation}\label{P-kernel}
K_P(x,y)=\int_{N_P(\Q)\bs N_P(\A)}\sum_{\gamma\in P(\Q)}h(x^{-1}\gamma ny)\, dn
\end{equation} 
For $T\in\af^+_0$ Arthur's distribution is defined by
\[
J^T(h)=\int_{A_G(\R)^0\GL(n,\Q)\bs\GL(n,\A)/K_f}\sum_P (-1)^{n-\dim A_P}K_P(x,x)
\widehat\tau_P(H_P(x)-T_P)dx,
\]
where $P$ runs over all $\Q$-rational parabolic subgroups of $\GL(n)$ and
the truncation parameter $T_P$ is chosen in such a way that
\[
\Ad(\delta)(H_P(x)-T_P)=H_{\delta P\delta^{-1}}(x)-T_{\delta P\delta^{1}}
\]
for all $\delta\in\GL(n,\Q)$. Note that this definition differs from the usual
definition, but it is easy to check that it agrees with the usual definition.
Furthermore, for $d(T)>d_0$, the sum over $P$ is finite. Using the decomposition
\eqref{isom9}, it follows that 
\[
J^T(h)=\sum_{j}\int_{\Gamma_{g_j,K_f}\bs\SL(n,\R)}\sum_P (-1)^{n-\dim A_P} K_P(g_jx,g_jx)
\widehat\tau_P(H_P(g_jx)-T_P)\, dx.
\]
Now assume that $h=h_\infty\otimes\chi_{K_f}$. Then the integrand 
$f(g_j^{-1}x^{-1}\gamma n xg_j)$ in $K_P(g_jx,g_jx)$ is nonzero, only if
\[
\gamma n\in \left(P(\Q) N_P(\A)\right)\cap \left(\GL(n, \R)
\cdot g_jK_f g_j^{-1}\right).
\]
We may decompose the integral 
\eqref{P-kernel} into a sum over 
\[
\gamma\in P_j(\Q):= P(\Q)\cap g_j K_f g_j^{-1}
\]

and an integral over 
\[
n\in P_j(\Q)\backslash  \left(P(\Q) N_P(\A)\right)\cap \left(\GL(n, \R)\cdot g_jK_f g_j^{-1}\right)
\simeq 
N_j(\Q)\bs N_j(\A)
\]
with $N_j(\Q)= N_P(\Q)\cap g_j K_f g_j^{-1}$ and $N_j(\A)
= N_P(\A)\cap g_j K_f g_j^{-1}$. Let $\widetilde P= P(\R)\cap \SL(n, \R)$.
Then $P_j(\Q)\cap \SL(n, \R)= \tilde P \cap \Gamma_{g_j, K_f}$ so that we get  
\[
J^T(h)=\sum_{j=1}^r\int_{\Gamma_{g_r,K_f}\bs\SL(n,\R)}\sum_P (-1)^{n-\dim A_P} K_{P,g_j,K_f}(x,x)
\widehat\tau_{P}(H_{P}(x)-T_{P})dx,
\]
where 
\[
K_{P,g_j,K_f}(x,y)=\int_{(\Gamma_{g_j,K_f}\cap N_P(\R))\bs N_P(\R)}
\sum_{\gamma\in\Gamma_{g_j,K_f}\cap \tilde P}
h_\infty(x^{-1}\gamma ny) dn.
\]

Now let $K(N)\subset \GL(n,\A_f)$ be the principal congruence subgroup of 
level $N$. Let $\Gamma(N)$ denote the principal congruence subgroup of level $N$ in $\SL(n,\Z)$, and let $\varphi(N)=\#\left(\Z/N\Z\right)^\times$ be the Euler function. Then $r=\varphi(N)$ and $\Gamma_{g_j, K(N)}\simeq \Gamma(N)$ for every $j$, cf.\ \cite[\S 4]{LM}.
Hence for $h=h_\infty\otimes\chi_{K_f}$ it follows that
\[
J^T(h)=\varphi(N)\int_{\Gamma(N)\bs\SL(n,\R)}\sum_P (-1)^{n-\dim A_P} K_{P,N}(x,x)
\widehat\tau_P(H_P(x)-T_P)dx,
\]
where
\[
K_{P,N}(x,y)=\int_{\Gamma(N)\cap N_P(\R)\bs N_P(\R)}\sum_{\gamma\in\Gamma(N)\cap \tilde P}
h_\infty(x^{-1}\gamma ny)dn.
\]
Let 
\[
Y(N)=A_G(\R)^0\GL(n,\Q)\bs \GL(n,\A)/K(N)
\]
and
\[
X(N)=\Gamma(N)\bs\SL(n,\R)/\SO(n).
\]
Then $Y(N)$ is the disjoint union of $\varphi(N)$ copies of $X(N)$. 
Let $\Delta_{p,X(N)}(\tau)$ be the Laplace operator on $E_\tau$-valued 
$p$-forms on $X(N)$. Then it follows from the definition of the regularized
trace \eqref{reg-trace7} that
\[
\Tr_{\reg}\left(e^{-t\Delta_{p,Y(N)}(\tau)}\right)=
\varphi(N)\Tr_{reg}\left(e^{-t\Delta_{p,X(N)}(\tau)}\right)
\]
for all $N\ge 3$. Using the definition \eqref{analtor} of the analytic
torsion, we obtain
\begin{equation}\label{analtor10}
\log T_{Y(N)}(\tau)=\varphi(N)\log T_{X(N)}(\tau).
\end{equation}
Furthermore we have 
\begin{equation}\label{volume1}
\vol(Y(N))=\varphi(N)\vol(X(N)).
\end{equation}
Combining \eqref{limittor3}, \eqref{analtor10}, and \eqref{volume1}, we obtain
the first part of Theorem \ref{theo-main}. The second part follows immediately
from \cite[Proposition 5.2]{BV}.

\end{document}